\documentclass[12pt]{article}
\usepackage{amssymb}
\usepackage{enumerate}
\usepackage{amsmath}
\usepackage{mathrsfs}
\usepackage{graphics}
\usepackage{graphicx}
\usepackage{xcolor}
\usepackage{pdfsync}
\usepackage{subfigure}
\usepackage[T1]{fontenc}
\usepackage{yhmath}
\usepackage{latexsym,amssymb,amsmath,amsfonts,amsthm}\usepackage{txfonts}
\numberwithin{equation}{section}
\newcommand{\be}{\begin{eqnarray}}
\newcommand{\ben}{\begin{eqnarray*}}
\newcommand{\en}{\end{eqnarray}}
\newcommand{\enn}{\end{eqnarray*}}

\newtheorem{theorem}{Theorem}[section]
\newtheorem{lemma}{Lemma}[section]

\newtheorem{thm}[theorem]{Theorem}

\newtheorem{remark}{Remark}

\begin{document}
\renewcommand{\theequation}{\arabic{section}.\arabic{equation}}
\begin{titlepage}
\title{\bf Large time behavior of strong solutions for  stochastic Burgers equation with transport noise}
\author{Zhao Dong$^{1,2}$,\ Feimin Huang$^{1,2}$, \ Houqi Su$^{1,2}$\\
	{\small	1. Academy of Mathematics and Systems Science, Chinese Academy of Sciences, Beijing 100190, China}\\
{\small 2. School of Mathematical Sciences, University of Chinese Academy of Sciences, Beijing 100049, China}\\
{\small {\sf dzhao@amt.ac.cn},\ {\sf fhuang@amt.ac.cn}, \ {\sf marksu@amss.ac.cn} }}
\date{}
\end{titlepage}
\maketitle

\noindent\textbf{Abstract}:
We consider the large time behavior of strong solutions to the stochastic Burgers equation with transport noise. It is well known that both the rarefaction wave and viscous shock wave are time-asymptotically stable for deterministic Burgers equation since the pioneer work of A. Ilin and O. Oleinik \cite{Olinik64} in 1964. However, the stability of these wave patterns under stochastic perturbation is not known until now. In this paper, we give a definite answer to the stability problem of the rarefaction and viscous shock waves for the 1-d stochastic Burgers equation with transport noise. That is, the rarefaction wave is still stable  under white noise perturbation and the viscous shock is not stable yet. 
Moreover, a time-convergence rate toward the rarefaction wave is obtained. To get the desired decay rate, an important inequality (denoted by Area Inequality) is derived. This inequality plays essential role in the proof, and may have applications in the related problems for both  the stochastic and deterministic PDEs.  

\



\section{Introduction}
The one dimensional Burgers equation reads
\begin{eqnarray}\label{1.1}
  \begin{array}{ll}
  u_t + uu_x =\nu u_{xx}, \\
  \end{array}
\end{eqnarray}
where the viscosity coefficient $\nu$ is a positive constant, and becomes the inviscid Burgers equation as $\nu=0$, i.e., 
\begin{eqnarray}\label{Hopf}
  \begin{array}{ll}
  u_t + uu_x = 0. \\
\end{array}
\end{eqnarray}
It is well known that the inviscid Burgers equation \eqref{Hopf} has rich wave phenomena such as shock and rarefaction wave, cf. \cite{Smoller}. Consider the Riemann initial data 
\begin{eqnarray}\label{Riemann}
u(x,0) = \left \{
\begin{array} {ll}
	u_- ,\quad x < 0,\\
	u_+, \quad x > 0, 
\end{array}
\right.
\end{eqnarray}
then the inviscid Burgers equation \eqref{Hopf} admits shock or rarefaction wave depending on the sign of $u_--u_+$. If $u_-<u_+$,  the solution of \eqref{Hopf} is rarefaction wave given by
\begin{eqnarray}\label{rarefaction}
u^r(t,x)=u^r\left( \frac{x}{t} \right) = \left \{
	 \begin{array} {ll}
	 	u_-, \quad &  x  < u_-t,\\
	 	\frac{x}{t}, \quad & u_-t <  x  < u_+t,\\
	 	u_+, \quad &  x  > u_+t,\\
	 \end{array}
\right.	
\end{eqnarray}
and if $u_->u_+$, the solution is shock wave given by 
\begin{eqnarray}\label{shock}
u^s(t,x) = \left \{
\begin{array} {ll}
u_-, \quad &  x  < x-st,\\
u_+, \quad &  x  > x-st,\\
\end{array}
\right.	
\end{eqnarray}
where $s=\frac{u_++u_+}{2}$ is the propagation speed of shock due to Rankine-Hugoniot (RH) condition.
 Both the shock and rarefaction waves are nonlinearly stable for the equation \eqref{Hopf}, cf. \cite{Smoller}. Since the Burgers equation \eqref{1.1} is a viscous version of \eqref{Hopf}, it is commonly conjectured that the viscous fundamental wave patterns, i.e., viscous shock wave and rarefaction wave, are stable for the Burgers equation \eqref{1.1}. The conjecture was first verified for one space dimension in \cite{Olinik64}, see also \cite{H&N91}, and \cite{hx2}, \cite{kawashima}, \cite{Ito96}, \cite{Nikishiwa99}, \cite{Xin90} for multi-dimensional scalar viscous conservation laws.  The conjecture is also valid for general systems of viscous conservation laws such as the compressible Navier-Stokes equations, cf. \cite{Huang-Xin-Yang}, \cite{LZ3},  \cite{MN85}-\cite{M&N86},  \cite{Z} and the Boltzmann equation, cf. \cite{Huang-Xin-Yang}, \cite{Liu-Yang-Yu-Zhao}, \cite{Liu-Yu} and the references therein, in which new techniques like weighted characteristic energy method, approximate Green's function and Evans function approach were developed. Nevertheless, we would ask a natural question: 

\

{\bf (Q):} Would shock and rarefaction waves be still stable under the stochastic perturbation?

\

From the physical point of view, these basic wave patterns might be perturbed by stochastic noise in the real environment. As a starting point, we focus on  the following stochastic Burgers equation with transport noise, 
\begin{eqnarray}\label{stochastic}
  \begin{array}{ll}
  du + uu_xdt =\mu u_{xx}dt + \sigma u_x dB(t),
  \end{array}
\end{eqnarray}
where $B(t)$ is one-dimensional standard Brownian motion on some  probability space $(\Omega,{\cal F}_t, P )$.  The stochastic term $\sigma u_x dB(t)$  can be explained as follows.  Let $u(t,x)$ be the smooth solution of the deterministic Burgers equation \eqref{1.1}.  
When the position $x$ is perturbed by a Brownian motion $\sigma B(t)$,  then $u(t,x+\sigma B(t))$ satisfies the equation \eqref{stochastic} due to It\^{o} formula,
where $\mu=\nu+\frac12\sigma^2$, see \cite{f2015} for the details,  see also \cite{Lillo94}, \cite{es} and \cite{Javier04}. The  shock formation, local ($\nu=0$) and global existence ($\nu>0$) of smooth solutions to \eqref{stochastic}  was investigated in \cite{ABT19}. 
The regularity effect of  noise recently attracts considerable attentions, see the interesting papers \cite{fgp2010, flandoli2021high, g2020, cg2019, gm2018, gs2019, gs2017, gj2013} and the references therein. 

In this paper, we focus on the large time behavior of strong solutions toward the viscous shock and rarefaction waves, and try to give a definite answer to the above question {\bf (Q)} for the stochastic Burgers equation \eqref{stochastic}. Roughly speaking, we show that the rarefaction wave is still stable  under transport noise perturbation and the viscous shock is not stable yet. Now we formulate the main results. 

Since the rarefaction wave $u^r(\frac{x}{t})$ given in \eqref{rarefaction} is only Lipschitz continuous, we follow the method of \cite{MN92} to introduce approximate rarefaction wave, which is a smooth solution of the following problem 
\begin{eqnarray}\label{app}
\left\{
\begin{array} {ll}
\bar{u}_t+\bar{u}\bar{u}_x=0,\\
\displaystyle \bar{u}(0,x)=\frac{u_{+}+u_{-}}{2}+ \frac{u_{+}-u_{-}}{2}k\int_{0}^{x}(1+\xi^2)^{-1}d\xi\to u_\pm,~\mbox{as}~x\to \pm\infty,
\end{array}
\right.	
\end{eqnarray}
where $\displaystyle k=\left(\int_{0}^{+\infty}(1+\xi^2)^{-1}d\xi\right)^{-1}$. The main result is stated as follows, 
\begin{theorem}[{\bf Rarefaction wave}]\label{1d}
Let  $\sigma^2<2\mu$ and $u_0(x)$ be the initial data of the stochastic Burgers equation \eqref{stochastic}. Set $\phi(t,x)=u(t,x)-\bar{u}(t,x)$.  If $\phi(0,x))\in\mathbb H^{2}(\mathbb R)$, then there exists a unique strong solution of \eqref{stochastic} satisfying
 \begin{eqnarray}\label{l2-1}
\mathbb E \| u(t,\cdot) - u^r(\cdot,t) \|_{\mathbb L^p (\mathbb R)}  \leq  C_p(2+t)^{-\frac{p-2}{4p}} \ln(2+t), \quad \forall p \in [2,+\infty),
\end{eqnarray}
and 
 \begin{eqnarray}\label{rate}
 \mathbb E \| u(t,\cdot) - u^r(t,\cdot) \|_{\mathbb L^\infty (\mathbb R)}  \leq  C_\epsilon (2+t)^{-\frac14+\epsilon},~\forall \epsilon>0, 
 \end{eqnarray}
where $u^r(t,x)$ is the rarefaction wave given in \eqref{rarefaction}. Moreover, it holds that for any $\epsilon>0$,  there exists a $\mathscr F_\infty$ measurable random variable  $C_\epsilon(\omega)\in \mathbb L^2(\Omega)$ such that 
  \begin{eqnarray}\label{rate for as}
   \| u(t,\cdot) - u^r(t,\cdot) \|_{\mathbb L^\infty (\mathbb R)}  \leq  C_\epsilon(\omega) (2+t)^{-\frac14+\epsilon}, ~a.s.
   \end{eqnarray}
\end{theorem}
\begin{remark}
Theorem \ref{1d} answers the question {\bf (Q)} in the case of rarefaction wave for the stochastic Burgers equation \eqref{stochastic}, i.e., the rarefacton wave is nonlinearly stable under transport noise. 
\end{remark}
\begin{remark}
The time-decay rate \eqref{l2-1} in $\mathbb L^p$ norm is almost optimal! Indeed, even for the deterministic heat equation 
\begin{eqnarray}\label{heat}
u_t=u_{xx}, ~u(0,x)\in L^2(\mathbb R),
\end{eqnarray}
the optimal decay rate of $u(t,x)$ in ${\mathbb L}^p$ is $(2+t)^{-\frac{p-2}{4p}}$.  In fact, the term $ln(2+t)$ in \eqref{l2-1} is 
coming from the  Brownian motion $u_xdB(t)$. 
\begin{remark}
	The assumption $\sigma^2<2\mu$ is equivalent to $\nu>0$ which is the viscosity of the deterministic Burgers equation \eqref{1.1}. Hence the assumption $\sigma^2<2\mu$ is necessary. 
\end{remark}
%

\end{remark}

The proof of Theorem \ref{1d} relies on a key inequality below, denoted by Area Inequality. 
 \begin{theorem}[{\bf Area Inequality}]\label{area} 
 		Assume that a Lipschitz continuous function $f(t)\ge 0$  satisfies 
	\begin{eqnarray}\label{lemma1}
	f^\prime(t)  \leq  C_0(1+t)^{-\alpha},
	\end{eqnarray}
	and
	\begin{eqnarray}\label{con}
	\int_0^t f(s)ds  \le   C_1(1+t)^\beta\ln^\gamma (1+t), ~\gamma\ge 0,
	\end{eqnarray}
for some constants $C_0$ and $C_1$, where $0\le\beta<\alpha$.	Then it holds that if $\alpha+\beta<2$, 
	\begin{equation}\label{l1}
	f(t) \le 2\sqrt{C_0C_1}(1+t)^{\frac{\beta-\alpha}{ 2}}\ln^{\frac{\gamma}{2}}(1+t),\quad \mbox{as}\quad t>>1. 
	\end{equation} 
 Moreover, if $\beta=\gamma=0$, i.e., $f(t)\in L^1[0,\infty)$ and $ 0<\alpha\le 2$, it holds that
 \begin{equation}\label{lemma2}
 f(t)= o(t^{-\frac{\alpha}{2}}), ~~\mbox{as}~~t>>1,
 \end{equation}
and the index $\frac{\alpha}{2}$ is optimal.
 \end{theorem}
\begin{remark}
The time-decay rate \eqref{l1} is surprising even for the case  $0<\alpha<1,~ \beta=\gamma=0$, in which the condition \eqref{con} becomes 	
\begin{eqnarray}\label{l3}
\int_0^{+\infty} f(t)dt  \le   C_1<+\infty.
\end{eqnarray}
To get the decay rate of $f(t)$, the usual way (see \cite{hl}) is to multiply \eqref{lemma1} by $1+t$,  then we have
	\begin{eqnarray}\label{lemma3}
[(1+t)f(t)]'\le f(t)+C_0(1+t)^{1-\alpha}.
	\end{eqnarray}
Integrating \eqref{lemma3} on $[0,T]$ implies that 
\begin{eqnarray}\label{lemma4}
(1+t)f(T)\le f(0)+\int_0^Tf(t)dt+C_0\int_0^T(1+t)^{1-\alpha}dt\le C+\frac{C_0}{2-\alpha}(1+t)^{2-\alpha},
\end{eqnarray}	
which gives 
\begin{eqnarray}\label{lemma5}
f(t)\le C(1+t)^{1-\alpha}, ~\mbox{as}~t>>1.
\end{eqnarray}	
It is impossible from  \eqref{lemma5} to get the time-decay rate  as $0<\alpha<1$ through the usual way \eqref{lemma3} to \eqref{lemma5}.  
Also note that for $1<\alpha<2$, the decay rate $t^{-\frac\alpha 2}$ in \eqref{lemma2} is faster than $t^{1-\alpha}$ in \eqref{lemma5}. 
\end{remark}
\begin{remark}
Since the inequality \eqref{lemma1} might be derived only for some $0<\alpha<1$ in the stability analysis, where $f(t)$ usually corresponds to the norm of some Sobolev spaces, we can expect that the Area Inequality might have  applications in the time-decay rate of solutions for both the deterministic and stochastic PDEs, see recent work \cite{hx1}. 
\end{remark}

In the case that $u_->u_+$, let $u^s(t,x):=\tilde{u}(\xi),~ \xi=x-st$ be the viscous shock wave of the deterministic Burgers equation \eqref{1.1} satisfying 
\begin{eqnarray}\label{vs}
\left\{
\begin{array} {ll}
-s\tilde{u}'+\tilde{u}\tilde{u}'=\nu \tilde{u}'',\\
\tilde{u}(\xi)\to u_\pm,~\mbox{as}~\xi\to \pm\infty,
\end{array}
\right.	
\end{eqnarray}
where $'=\frac{d}{d\xi}$, $s=\frac{u_-+u_+}{2}$. Without loss of generality, let $s=0$, i.e., $u_-=-u_+>0$. 
It is known in \cite{Olinik64} that the equation \eqref{vs} admits a unique solution $\tilde{u}(\xi)$ up to a shift. Since the position $x$ is perturbated by $\sigma B(t)$, the perturbed viscous shock 
is $\tilde{u}^B(t,x):=\tilde{u}(x+\sigma B(t))$ and the two viscous shock waves coincide at the initial time, i.e., $\tilde{u}(x)=\tilde{u}^B(0,x)$. Let $$d(t)=\mathbb E\|\tilde{u}(x)-\tilde{u}^B(t,x)\|_{L^\infty(\mathbb R)}, ~d(0)=0.$$
We have the following instability theorem. 
\begin{theorem}[Instability for shock wave]\label{instability}
	 $d(t)$ is an increasing function of $t$. Moreover, it holds that 
	\begin{equation}
	\lim_{t\to +\infty}d(t)=u_--u_+. 
	\end{equation}
	
\end{theorem}
\begin{remark}
	Theorem \ref{instability} indicates that the viscous shock wave is not stable under transport noise  perturbation. 
\end{remark}

\

We outline the proof of Theorem \ref{1d}. One of the main difficulties comes from the stochastic perturbation $u_xdB(t)$ in the whole space  $\mathbb R$ such that mild solution approach and some compact methods might not be available anymore. 
This difficulty is overcome by combining energy method, iteration approach and new $\mathbb L^p$ and martingale estimates. One of the advantages of energy method is that the stochastic integral term can be cancelled in the expectation, while it is not clear in the mild solution formula. Precisely speaking, 
we first apply a cut-off technique to control the nonlinear term $u^2$. Then we instead consider the cut-off equation \eqref{cutoff} whose global existence of strong solution is shown by the contraction mapping principle and energy method. It is noted that the quadratic variation for the derivative is understood in Krylov's theory \cite{Kr}.   
Once the global existence of \eqref{cutoff} is obtained,  the global existence of strong solution to the original equation \eqref{1.1} is proved through a new martingale estimate and stopping time method. 

Another difficulty is from the rarefaction wave.  The quadratic variation generates a bad term $\|\bar{u}_x\|^2\approx \frac{1}{2+t}$ which is not integrable over $[0,\infty)$.  By the energy method,  $\mathbb E \| \phi(t,\cdot)\|_{\mathbb L^2 (\mathbb R)}$ may increase with time $t$ (probably $\ln^\frac12 (2+t)$), while $ \| \phi(t,\cdot)\|_{\mathbb L^2 (\mathbb R)}$ is uniformly bounded for the deterministic Burgers equation \eqref{1.1}). To get the desired a priori estimates, the time-decay rate of $\mathbb E \| \phi_x(t,\cdot)\|_{\mathbb L^2 (\mathbb R)}$ for the derivative is necessarily acheived. Fortunately, we observe that for any $2<p<+\infty$, $\mathbb E \| \phi(t,\cdot)\|_{\mathbb L^p (\mathbb R)}$ decays with a rate by a new $L^p$ energy method and martingale estimates although the $L^2$ norm may increase. We further observe that $\mathbb E \| \phi(t,\cdot)\|_{\mathbb L^p (\mathbb R)}$ provides a  time-decay rate with some $0<\alpha<1$ in the energy inequality for $\phi_x$ so that the Area Inequality can be applied, where $f(t)=\mathbb E \| \phi_x(t,\cdot)\|^2_{\mathbb L^2 (\mathbb R)}$. The time-decay rate of  $\mathbb E \| \phi(t,\cdot)\|_{\mathbb L^\infty (\mathbb R)}$ in \eqref{rate} is derived by the celebrated Gargliado-Nirenberg inequality.  
Finally, the decay rate \eqref{rate for as} a.s. is obtained by combining the martingale estimates and the decay \eqref{rate} in expectations.

\

For the other works of the Burgers equation and conservation laws with stochastic force, see the interesting papers \cite{cdk,ddt,dg,dv1,dv2,dsxz2018,fn,gjs2015,gj2013,hm2016,hw2013,pz2020} and the references therein.

\

The rest of the paper is organized as follows. Sections \ref{pre}-\ref{global existence} are devoted to the proof of Theorem \ref{1d}. Among them, some preliminaries 
on the approximate rarefaction waves are given in section \ref{pre}, while the Area Inequality is proved in section \ref{area inequality}. Section \ref{a priori estimates} is devoted to the a priori estimates through a delicate $L^p, p\ge 2$ energy method and then the time-decay rates \eqref{rate} and \eqref{rate for as} of strong solution toward the rarefaction wave are given in section \ref{decay estimates}.   In Section \ref{global existence}, the global existence of strong solution is proved by combining the martingale estimates  and the global existence of the cut-off equation \eqref{cutoff} through the contraction mapping principle and stopping time method. Finally the instability of viscous shock wave is given in Section \ref{instability1}.

\section{Preliminaries}\label{pre}

In this section, some properties of the approximate rarefaction wave $\bar{u}(t,x)$ given in \eqref{app} are listed as follows. 
\begin{lemma}[\cite{MN92}, \cite{Ito96}]\label{ito1996}
	\begin{enumerate}[i)]
		\item $u_- < \bar{u}(t,x) < u_+,~ \bar{u}_x(t,x) > 0$ for $(t,x) \in \mathbb R^+ \times \mathbb R$,
		\item For all $p\in [1,\infty] $, there exist constants $C_p$ and $C_{p,d}$ such that for large $t$,  
		\begin{eqnarray}\begin{array}{lll}
			\| \bar{u}_x(t,\cdot) \|_{\mathbb L^p(\mathbb R)} &\leq & C_p \min (d, \ d^{1/p} t^{-1+1/p}),\\
			\| \bar{u}_{xx}(t,\cdot) \|_{\mathbb L^p(\mathbb R)} & \leq & C_p \min (d,\ d^{-(p-1)/2p} t^{-(1+(p-1)/p)}),\\
			\| \bar{u}_{xxx}(t,\cdot) \|_{\mathbb L^p(\mathbb R)} & \leq & C_{p,d} (2+t)^{-(1+(2p-1)/2p)},
			\end{array}
		\end{eqnarray}
		where $d=u_+-u_-$ is the strength of rarefaction wave. 
		\item For all $p\in(1,\infty]$, there is a constant $C_{p,d}$ such that for large $t$, 
		\begin{eqnarray}\label{difference}
				\| \bar{u}(t,\cdot) - u^r(t,\cdot) \|_{\mathbb L^p(\mathbb R)} \leq C_{p,d} t^{-(p-1)/2p},
		\end{eqnarray}
where $u^r(t,x)$ is the rarefaction wave given in \eqref{rarefaction}.
	\end{enumerate}
\end{lemma}
Next, we give the Gagliardo-Nirenberg (GN) inequality 
which reads as, for any $ 1\leq p\leq +\infty $ and integer $ 0\leq j < m $, 
\begin{equation}\label{gn}
\|{\nabla_x^j u}\|_{L^p(\mathbb R^n)} \leq C \|{\nabla_x^m u}\|_{L^r(\mathbb R^n)}^{\theta} \|{u}\|_{L^q(\mathbb R^n)}^{1-\theta},
\end{equation}
where $ \frac{1}{p} = \frac{j}{n} + \left(\frac{1}{r} -\frac{m}{n} \right) \theta + \frac{1}{q} \left(1-\theta \right) $, $ \frac{j}{m} \leq \theta \leq 1$. Based on the GN inequality, we also list an interesting interpolation inequality, which will be used later. 
\begin{lemma}[\cite{kawashima}]\label{ii} It holds that, 
	\begin{align}\label{ii1}
	\|u\|_{L^p}\leq C(p,q,n)\|\nabla(|u|^{\frac{p}{2}})\|^{\frac{2\gamma}{1+\gamma p}}_{L^2}\|u\|^{\frac{1}{1+\gamma p}}_{L^q}
	\end{align}
	for $2\leq p<\infty$ and $1\leq q\leq p$, where $\gamma=(n/2)(1/q-1/p)$, and $C(p,q,n)$ is a positive constant.
\end{lemma}
\section{Area inequality}\label{area inequality}
 This section is devoted to the Area Inequality, which plays a key role to get the convergence rate of strong solution toward the rarefaction wave for the stochastic Burgers equation \eqref{stochastic}. 
 \begin{proof}[\bf Proof of Theorem \ref{area}]
We first prove \eqref{l1}, i.e., 
	\begin{equation}\label{decayrate}
	f(t) \le 2\sqrt{C_0C_1}(1+t)^{\frac{\beta-\alpha}{ 2}}\ln^{\frac{\gamma}{2}}(1+t),~t>>1,
	\end{equation}
under the conditions 
 	\begin{eqnarray}\label{cor1}
 	f^\prime(t)  \leq  C_0(1+t)^{-\alpha},
 	\end{eqnarray}
 	and \begin{eqnarray}\label{cor2}
 \int_0^t f(s)ds  \le   C_1(1+t)^\beta\ln^\gamma (1+t), ~\gamma\ge 0,
 	\end{eqnarray}
 where  $0\le\beta<\alpha$ and $\alpha+\beta<2$. The proof is provided by the way of contradiction.

  If the inequalaity \eqref{decayrate} does not hold, then there exists a sequence  $ \{ t_n \}_{n=1}^{\infty}$ with $ t_n \uparrow \infty $ such that $ f(t_n) > C_2(1+t_n)^{\frac{\beta-\alpha}{ 2}}\ln^{\frac{\gamma}{2}}(1+t_n)$, where $C_2:=2\sqrt{C_0C_1}$. Note that   the inequality \eqref{cor1} in the interval $[0,t_n]$ is equivalent to 
  \begin{equation}
  \frac{df(\tilde{\tau})}{d\tilde{\tau}}\ge -C_0(1+t_n-\tilde{\tau})^{-\alpha}, ~f(\tilde{\tau})|_{\tilde{\tau}=0}=f(t_n), ~\tilde{\tau}\ge 0,
  \end{equation}
  where $\tilde{\tau}=t_n-\tau$.  Then we construct a function $g_n(\tilde{\tau})$ satisfying 
  \begin{equation}\label{ode}
  \frac{dg_n(\tilde{\tau})}{d\tilde{\tau}}= -C_0(1+t_n-\tilde{\tau})^{-\alpha}, ~g_n(\tilde{\tau})|_{\tilde{\tau}=0}=C_2(1+t_n)^{\frac{\beta-\alpha}{ 2}}\ln^{\frac{\gamma}{2}}(1+t_n), ~\tilde{\tau}\ge 0.
  \end{equation}
It is obvious that $f(\tau)\ge g_n(\tau)$ for any $\tau\in [0,t_n]$ since $f(\tilde{\tau})|_{\tilde{\tau}=0}\ge g_n(\tilde{\tau})|_{\tilde{\tau}=0}$, see Figure 1 below. 

\begin{center}
	\includegraphics[scale=0.7]{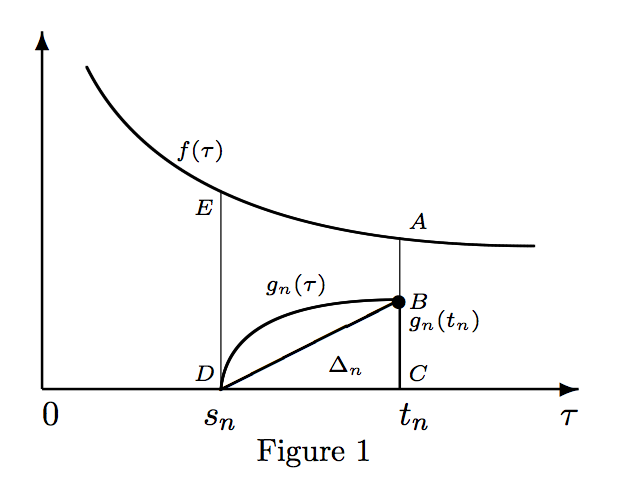}\\[0.1cm]
\end{center}
  The ODE \eqref{ode} is exactly 
  a backward ordinary differential equation starting from $t_n$, i.e., 
  \begin{equation}\label{o1}\left\{\begin{array}{l}
  \displaystyle\frac{dg_n(\tau)}{d\tau}=C_0(1+\tau)^{-\alpha}, ~0\le \tau \le t_n,\\
 g_n(t_n):=C_2(1+t_n)^{\frac{\beta-\alpha}{ 2}}\ln^{\frac{\gamma}{2}}(1+t_n). 
  \end{array}
  \right.
  \end{equation}
A direct computation gives the formula of $g_n(\tau)$ for $\alpha\ne 1$ on $0<\tau\le t_n$, 
\begin{equation}\label{o1-1}
g_n(\tau)=g_n(t_n)-C_0\int_{\tau}^{t_n}(1+s)^{-\alpha}ds=g_n(t_n)-\frac{C_0}{1-\alpha}[(1+t_n)^{1-\alpha}-(1+\tau)^{1-\alpha}].
\end{equation}
Due to $0\le \beta<\alpha$ and $ \alpha+\beta<2$, $\frac{\beta-\alpha}{2}<1-\alpha$ is always true. 
Taking $\tau=\frac{t_n}{2}$ gives that 
\begin{equation}\label{3.8}
g_n(\frac{t_n}{2})=g_n(t_n)-\frac{C_0}{1-\alpha}[(1+t_n)^{1-\alpha}-(1+\frac{t_n}{2})^{1-\alpha}]<0, ~~\mbox{as}~~t_n>>1,
\end{equation}
Since $g_n(\tau)$ is monotonically increasing, there exists a unique $s_n\in (\frac{t_n}{2},t_n)$ such that $g_n(s_n)=0$. 
Taking $\tau=s_n$ in \eqref{o1-1},  it follows  from the mean value theorem that 
\begin{equation}\begin{array}{lll}\label{o2}
s_n&=&[(1+t_n)^{1-\alpha}-\frac{1-\alpha}{C_0}g_n(t_n)]^{\frac{1}{1-\alpha}}-1\\
&=&(1+t_n)[1-\frac{1-\alpha}{C_0}g_n(t_n)(1+t_n)^{\alpha-1}]^{\frac{1}{1-\alpha}}-1\\
&=&(1+t_n)[1-\frac{1}{C_0}g_n(t_n)(1+t_n)^{\alpha-1}(1-\xi_{n})^{\frac{\alpha}{1-\alpha}}]-1\\
&\le&t_n-\frac{2}{3C_0}g_n(t_n)(1+t_n)^{\alpha},
\end{array}
\end{equation}
where $\xi_{n}\in (\min\{0,1-(\frac23)^{\frac{\alpha}{1-\alpha}}\},\max\{0,1-(\frac23)^{\frac{\alpha}{1-\alpha}}\})$ is a constant, and we have used the fact that
 \begin{equation}\label{inf}
g_n(t_n)(1+t_n)^{\alpha-1}=C_2(1+t_n)^{\frac{\alpha+\beta}{2}-1}\ln^\frac{\gamma}{2}(1+t_n)=o(1),~~\mbox{as}~~t_n\to \infty.
\end{equation} 
Then we have 
	\begin{equation}\label{o3}
	t_n-s_n\ge \frac{2}{3C_0}g_n(t_n)(1+t_n)^{\alpha}.
	\end{equation}

Note that the curve $g_n(\tau)$ is concave due to the fact that the derivative of $(1+\tau)^{-\alpha}$ in \eqref{o1} is negative. Thus the region $S_{ACDE}$ surrounded by the segments $\overline{AC}$, $\overline{CD}$,  $\overline{DE}$ and the curve $\widetilde{EA}$ should cover the triangle $\triangle${BCD} (this is why we call Theorem \ref{area} by Area Inequality), see Figure 1. 
From \eqref{cor2} and \eqref{o3}, a direct computation gives
	\begin{equation}\label{area1}
	\begin{array}{ll}
\displaystyle	~~~~C_1(1+t_n)^\beta\ln^\gamma (1+t_n)\ge \int_0^{t_n}f(\tau)d\tau\ge  \int_{s_n}^{t_n}f(\tau)d\tau=\mbox{Area of $S_{ACDE}$}\\
\displaystyle	\ge \mbox{Area of $S_{\Delta BCD}$}=\frac12(t_n-s_n)g_n(t_n)
\ge 
 \frac{C_2^2}{3C_0}(1+t_n)^\beta\ln^\gamma (1+t_n), 
		\end{array}
	\end{equation}
	which impies $3\ge 4$ due to $C_2^2=4C_0C_1$. This is a contradiction and thus the inequality \eqref{decayrate} holds for $\alpha\ne 1$. 

 	For the case $\alpha=1$, the same argument implies that there exists a unique $s_n\in (\frac{t_n}{2},t_n)$ such that $g_n(s_n)=0$. Then the formula of $g_n(\tau)$ gives 
 	\begin{equation}\label{o5}
 	g_n(t_n)=g_n(s_n)+C_0\int_{s_n}^{\tau}(1+s)^{-1}ds=C_0\ln\frac{1+t_n}{1+s_n}.
 	\end{equation}
 Note that $g_n(t_n)=C_2(1+t_n)^{\frac{\beta-1}{2}}
 	\ln^{\frac{\gamma}{2}}(1+t_n)=o(1)$, 
 	we get 
 	\begin{equation}\begin{array}{lll}\label{o2-1}
 	s_n=(1+t_n)e^{-\frac{g_n(t_n)}{C_0}}-1\approx (1+t_n)(1-\frac{g_n(t_n)}{C_0})-1=t_n-\frac{1}{C_0}g_n(t_n)(1+t_n),
 	\end{array}
 	\end{equation}
 	which yields the inequality \eqref{o3}. As in \eqref{area1},  the same argument  implies \eqref{decayrate}.
 	
 	Next we consider the case that $\beta=\gamma=0$ and $0<\alpha\le 2$, i.e., 
 	\begin{equation}\label{integral}
 	\int_0^\infty f(t)dt<+\infty,
 	\end{equation}
and we expect a better result \eqref{lemma2}. If \eqref{lemma2} does not hold, then there exist a small constant $\epsilon$ and a sequence  $ \{ t_n \}_{n=1}^{\infty}$ with $ t_n \uparrow \infty $ such that $ f(t_n) > \epsilon (1+t_n)^{-\frac{\alpha}{ 2}}$. In the same way as in \eqref{o1} and \eqref{3.8}, we can construct a function $g_n(\tau)$ with $g_n(t_n):=\epsilon (1+t_n)^{-\frac{\alpha}{ 2}}$ and $g_n(s_n)=0$, $s_n\in (\frac{t_n}{2},t_n)$. Note that $g_n(t_n)(1+t_n)^{\alpha-1}
=\epsilon (1+t_n)^{\frac{\alpha}{2}-1}$ is small as $0<\alpha\le 2$ so that the inequality \eqref{o3} still holds for $t_n-s_n$.  
 By the same argument as in \eqref{area1}, we  have as $t_n>>1$, 
 	\begin{equation}\label{o4}
 	\int_{s_n}^{t_n}f(\tau)d\tau\ge \frac12 g_n(t_n)(t_n-s_n)\ge \frac{2}{3C_0} (g_n(t_n))^2(1+t)^{\alpha}=\frac{2\epsilon^2}{3C_0}.
 	\end{equation}
Note that the left hand side of \eqref{o4} 
tends to zero as $t_n\to +\infty$ due to \eqref{integral}, while the right one is a fixed constant. This is a contradiction and hence \eqref{lemma2} holds, i.e., $f(t)=o(t^{-\frac{\alpha}{2}})$, $t>>1$.
 	 	 	
 	
To prove Theorem \ref{area}, it remains to show that the index in the decay rate \eqref{lemma2}  is optimal. For this, we only need to prove that for any $\epsilon>0$, there exist a sequence  $ \{ t_n \}_{n=1}^{\infty}$ with $ t_n \uparrow \infty $ and  a function $g(t)$ satisfying all conditions of Theorem  \ref{area} for $\beta=\gamma=0$ such that $g(t_n)= (1+t_n)^{-\frac{\alpha}{2}-\epsilon}$. The function $g(t)$ is constructed as follows.  
	
Let $s_n=e^n$. As in \eqref{o1}, we consider the following ordinary differential equation on $[s_n,t_n]$, 
	\begin{eqnarray}\label{o6}
	\ g_n^\prime(\tau) =C_0(1+\tau)^{-\alpha},~	g_n(t_n) = (1+t_n)^{-\frac{\alpha}{2}-\epsilon}, \ g_n(s_n) = 0,
	\end{eqnarray}
which gives
\begin{equation}\label{o7}
t_n-s_n\approx \frac{1}{C_0}g_n(t_n)(1+t_n)^{\alpha} =\frac{1}{C_0}(1+t_n)^{\frac{\alpha}{2}-\epsilon},~~n>>1.
\end{equation} 
Let $g_n(t)$ monotonically decrease to zero on $[t_n,z_n]$ where $z_n$ is close to $t_n$ so that $z_n<s_{n+1}$ and 
\begin{equation}\label{o8}
\sum_{n=1}^\infty \int_{t_n}^{z_n}g_n(t)dt<\infty,
\end{equation}
see Figure 2 below.
	\begin{center}
		\includegraphics[scale=0.8]{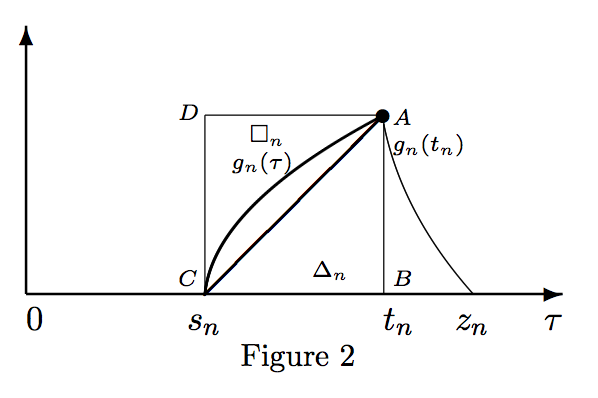}\\[0.1cm]
	\end{center}
Define
\begin{equation}\label{g}
g(t)=\left\{\begin{array}{ll}
g_n(t), t\in [s_n,z_n],\\
0, ~\mbox{otherwise}. 
\end{array}
\right.
\end{equation}
It is obvious to see that $g'(t)=C_0(1+t)^{-\alpha}$ on $[s_n,t_n]$ and $g'(t)\le 0$ as $t\notin [s_n,t_n]$ so that \eqref{lemma1} is satisfied. 

On the other hand, the integral $\int_{s_n}^{t_n}g(\tau)d\tau$ is less than the area of the rectangle $\square_nABCD$ with the width $t_n-s_n$ and the height $g_n(t_n)=(1+t_n)^{-\frac{\alpha}{2}-\epsilon}$, i.e., 
\begin{equation}\label{o9}
\int_{s_n}^{t_n}g(\tau)d\tau\le (t_n-s_n)g_n(t_n)\le\frac{2}{C_0}(1+t_n)^{-2\epsilon}\le\frac{2}{C_0}s_n^{-2\epsilon}\le  Ce^{-2n\epsilon},
\end{equation}
which implies, together with \eqref{o8},
\begin{equation}
\int_0^\infty g(t)dt=\sum_{n=1}^\infty\int_{s_n}^{t_n}g(t)dt+\sum_{n=1}^\infty\int_{t_n}^{z_n}g(t)dt\le C\sum_{n=1}^\infty e^{-2n\epsilon}+C<\infty. 
\end{equation}
%
In particular, $g(t_n) = (1+t_n)^{-\frac{\alpha}{2}-\epsilon}$. 
Thus the index in \eqref{lemma2} is optimal. Therefore Theorem \ref{area} is completed. 
\end{proof}

 
 


\section{The a priori estimates}\label{a priori estimates}
First we introduce a useful lemma concerning the regularity of strong solution to \eqref{stochastic}. 
\begin{lemma}[Theorem 4.10 and 5.1 in  \cite{Kr}]\label{rs}
	Let $f\in L^2([0,T]\times\Omega, \mathbb H^1(\mathbb R))$ and $g\in L^2([0,T]\times\Omega, \mathbb H^2(\mathbb R))$ be adapted process, then for  any $\phi_0\in H^2({\mathbb R})$ and any $T>0$, the stochastic equation 
	\begin{eqnarray} \label{random source}
	\left\{
	\begin{array}{ll}
	d \phi = \mu\phi_{xx}dt+\sigma \phi_xdB(t)+fdt+ gdB(t), \\
	\phi(0)=\phi_0,\\
	\end{array}
	\right.
	\end{eqnarray}
	with $\sigma^2< 2\mu$, admits 
	a  unique strong solution in 
	$$
W_T:=\Big\{\phi\in C\left( (0,T) ;\mathbb{H}^2(\mathbb{R}) \right)\cap  \mathbb L^2\left( (0,T);  \mathbb{H}^3(\mathbb{R}) \right)\Big\}.
	$$
\end{lemma}

Let $\phi(x,t) = u(t,x) - \bar u(t,x) $, then the stochastic Burgers equation \eqref{stochastic} can be  reduced into the following perturbation equation
\begin{eqnarray}\label{Perturbation}
\left\{
       \begin{array} {ll}
       	d\phi + ( \phi \bar u )_x dt +  \frac{1}{2} (\phi^2 )_x dt = \mu \phi_{xx} dt + \mu \bar u_{xx} dt + \sigma( \phi_x  + \bar u_x ) dB(t), \quad \mbox{in} ~~ \mathbb{R}\times [0,\infty), \\
           \phi|_{t=0}(x)=\phi_0(x), \quad \mbox{in} ~~\mathbb{R}.\\
              \end{array}
\right.	
\end{eqnarray}
The global existence of strong solution to  \eqref{Perturbation} will be proved in Section \ref{global existence} through a cut-off equation \eqref{cutoff}. The rest of this section is devoted to the a priori estimates of solutions in the following solution space 
	\[\label{X}
	X_T:= \left\{ \phi\in C\left( (0,T) ;\mathbb{H}^1(\mathbb{R}) \right), \phi_x\in \mathbb L^2\left( (0,T);  \mathbb{H}^1(\mathbb{R}) \right) \right\}\supset W_T.\]
 The norm $\|\cdot\|_T$ is defined as 
	\[
	\| \phi \|_T\triangleq \Bigg(\mathbb E\sup_{0\leq s\leq T}  \| \phi(s) \|_{\mathbb{H}^1(\mathbb{R})}^2 + \mathbb E \int_0^T \| \phi_x (s) \|_{\mathbb{H}^1(\mathbb{R})}^2 ds\Bigg)^{\frac12}.
	\]
We have
\begin{lemma}\label{lemma4.1}
Assume $\sigma^2 < 2\mu $ and $\phi(t,x)\in X_T$ is the strong solution of \eqref{Perturbation}, it holds that
\begin{eqnarray}	\label{3.14}
\begin{array}{lll}
&&	\displaystyle\| \phi(t) \|^2 + \int_0^t \| \phi_x \|^2 ds + \int_0^t \int_\mathbb R \phi^2 \bar u_x dx dt\\[3mm]
& \leq &\displaystyle C_1\left( \| \phi_0 \|^2+ \ln (2+t)\right) +C_2 \int_0^t \int_\mathbb R \phi  \bar u_x dxdB(t).
\end{array}
\end{eqnarray}

\end{lemma}
\begin{proof}
Multiply   (\ref{Perturbation}) by  $\phi$, it holds in the sense of It\^{o} integral that  
\begin{eqnarray}
	\phi d \phi + \phi^2 \bar u _x dt + \phi \phi_x \bar u dt + \phi^2 \phi _x dt = \mu \phi \phi _{xx} dt + \mu \phi \bar u_{xx} dt +  \sigma\phi (\phi_x + \bar u_x ) dB(t), 
\end{eqnarray}
i.e., 
\begin{eqnarray}
\begin{array}{lll}
&&\displaystyle\int_0^t\phi(s) d \phi(s) + \int_0^t\phi^2 \bar u _x ds + \int_0^t\phi \phi_x \bar u ds + \int_0^t\phi^2 \phi _x ds\\& =&\displaystyle \int_0^t\mu \phi \phi _{xx} ds +\int_0^t \mu \phi \bar u_{xx} ds +  \sigma\int_0^t\phi (\phi_x + \bar u_x ) dB(s).
\end{array}
\end{eqnarray}
Using It\^{o} formula and stochastic Fubini theorem yield that 
\begin{eqnarray}\label{4.6}
\begin{aligned}
&	\frac{1}{2} d \| \phi \|^2 -\frac12 d \langle \| \phi \|^2 \rangle_t + \frac{1}{2} \int_\mathbb R \phi^2 \bar u _x dx dt\\
 =& - \mu \| \phi_x \|^2 dt + \mu \int_\mathbb R \phi \bar u_{xx} dx dt + \sigma\int_\mathbb R  \phi\bar u_x  dx dB(t).
	\end{aligned}
	\end{eqnarray}
By Lemma 2.1 ii),  we obtain
\begin{eqnarray}
	\left|\int_\mathbb R \phi \bar u_{xx} dx \right|  =   \left| \int_\mathbb R \phi_x \bar u_{x} dx \right| \leq  \epsilon \| \phi_x \|^2 + C_\epsilon\| \bar u_x \|^2  \leq \epsilon \| \phi_x \|^2 + C_\epsilon (2+t)^{-1}.
\end{eqnarray}
In addition, from It\^{o} formula, the quadratic variation reads
\begin{eqnarray}
	d\langle \| \phi \|^2 \rangle_t = \| \sigma(\phi_x + \bar u_x) \|^2 dt \leq \sigma^2 ( 1+ \epsilon) \| \phi_x \|^2 dt +  C_\epsilon (2+t)^{-1}dt, ~\forall~ 0 < \epsilon < 1,
\end{eqnarray}
which gives that 
\begin{eqnarray}\label{gronwall}
\begin{array}{lll}
	&&\displaystyle\| \phi(t) \|^2 + \int_0^t \int_\mathbb R \phi^2 \bar u_x dx dt + (2\mu- \sigma^2- 4\epsilon)\int_0^t \| \phi_x \|^2 ds \\
    \leq &&\displaystyle\| \phi_0 \|^2  + C_{\epsilon} \ln(2+t)  + 2\sigma \int_0^t \int_\mathbb R \phi \bar u_x dxdB(t).
    \end{array}
\end{eqnarray}
Choosing $\epsilon$ small enough implies that 
\begin{eqnarray}\label{3.15}
\begin{array}{lll}
&&\displaystyle	\| \phi(t) \|^2 + \int_0^t \| \phi_x \|^2 ds + \int_0^t \int_\mathbb R \phi^2 \bar u_x dx ds\\
&\leq &\displaystyle C_1 \left(\| \phi_0 \|^2 + \ln (2+t)\right)  + C_2\int_0^t \int_\mathbb R \phi\bar u_x  dxdB(s).
\end{array}
\end{eqnarray}
\end{proof}

\begin{lemma}
Assume $\sigma^2 < 2 \mu $ and $\phi(t,x)\in X_T$ is the strong solution of \eqref{Perturbation}, it holds that
\begin{eqnarray} \label{3.16}
\begin{array}{lll}
&& \displaystyle	d\| \phi_x \|^2 + \| \phi_{xx} \|^2 dt + \int_\mathbb R \phi_x^2 \bar u_x dx dt\\ \leq  &&\displaystyle C_1 \left((2+t)^{-2} dt +(2+t)^{-2}\|\phi\|^2dt+ \| \phi \|_{\mathbb L^6(\mathbb R)}^6 dt\right) - C_2 \int_\mathbb R \phi_{xx}   \bar u_{x} dx dB(t).
\end{array}
 \end{eqnarray}
\end{lemma}

\begin{proof}
Multiply \eqref{Perturbation} by $-\phi_{xx} $,  it holds that, 
\begin{eqnarray*}
	-\phi_{xx} d \phi -\phi_{xx} (\phi_x \bar u + \phi \bar u_x) dt - \phi \phi_x \phi_{xx} dt = -\mu \phi_{xx}^2 dt - \mu \phi_{xx} \bar u_{xx} dt - \phi_{xx} \sigma (\phi_x + \bar u_x) dB(t).
\end{eqnarray*}
Note that from Lemma \ref{ito1996} ii, we have 
 \begin{eqnarray}
 	\int_\mathbb R -\phi_{xx} (\phi_x \bar u + \phi \bar u_x) dx &=&
 	\frac{1}{2} \int_\mathbb R \phi_x^2 \bar u_x dx - \int_ \mathbb R \phi_{xx} \phi \bar u_x dx \\\notag
 &\geq & \frac{1}{2} \int_\mathbb R \phi_x^2 \bar u_x dx - \epsilon \| \phi_{xx} \|^2 - C_\epsilon (2+t)^{-2}\| \phi \|^2,\\
\left|	\int_\mathbb R \phi_{xx} \bar u_{xx} dx\right| &\leq& \epsilon \| \phi_{xx} \|^2 + C_\epsilon (2+t)^{-2},\\
 	\left|		\int_\mathbb R \phi\phi_x\phi_{xx} dx\right|  &\le&\frac12\|\phi_x\|_{L^3(\mathbb R)}^3\le  C\|\phi_{xx}\|^{\frac32}\|\phi\|_{L^6(\mathbb R)}^\frac32\le \epsilon\|\phi_{xx}\|^2+C_\epsilon\|\phi\|_{L^6(\mathbb R)}^6.
 \end{eqnarray}
 On the other hand, from Lemma \ref{rs} 
 in which 	$f:= -(\phi \bar u )_x  - \frac{1}{2} (\phi^2 )_x\in L^2([0,T]\times\Omega, \mathbb H^1(\mathbb R))$ and $g:= \sigma\bar u_x\in L^2([0,T]\times\Omega, \mathbb H^2(\mathbb R))$, we conclude that $\|\phi_x\|^2$ is semi-martingale. Thus we get 
 \begin{eqnarray}
d\langle \| \phi_x \| ^2 \rangle_t = \sigma^2 \| \phi_{xx} + \bar u_{xx} \|^2 dt \leq  \sigma^2 ( 1 + \epsilon ) \| \phi_{xx} \|^2 dt + \sigma^2 C_{\epsilon} (2+t)^{-2} dt ,
 \end{eqnarray}
 which indicates that by choosing $\epsilon$ small, 
 \begin{eqnarray}\label{deriva}
 \begin{array}{lll}
 &&	\displaystyle d\| \phi_x \|^2 + \| \phi_{xx} \|^2 dt + \int_\mathbb R \phi_x^2 \bar u_x dx dt \\
 \leq &&\displaystyle C_1 \left((2+t)^{-2} dt +(2+t)^{-2}\|\phi\|^2dt+ \| \phi \|_{\mathbb L^6(\mathbb R)}^6 dt\right) -C_2  \int_\mathbb R \phi_{xx}  \bar u_{x}dx dB(t).
 \end{array}
 \end{eqnarray}
\end{proof}

Next we give an interesting $\mathbb L^p(\mathbb R)$ estimate which is crutial to close the a priori estimates.
\begin{lemma}\label{lemma4.3}
 Assume $\sigma^2 < 2\mu$ and $\phi(t,x)\in X_T$ is the strong solution of \eqref{Perturbation}, it holds that for any $p > 2$, 
 \begin{equation}\label{3.17}
d\| \phi(t) \|_ {\mathbb L^p(\mathbb R)} ^p +  \int_\mathbb R |\phi|^p \bar u_x + |\phi|^{p-2}\phi_x^2 dx dt \leq C_1( 2+ t)^{-\frac{p}{2}}dt + C_2\int_ \mathbb R |\phi|^{p-2} \phi \bar u_x  dx dB(t).
\end{equation}
\end{lemma}
 \begin{proof}
 Multiply  \eqref{Perturbation} by $|\phi|^{p-2}\phi$, we have 
 \begin{eqnarray}
\begin{array}{lll}&& \displaystyle|\phi|^{p-2}\phi d \phi + |\phi|^{p-2}\phi (\phi \bar{u})_x dt + |\phi|^{p-2}\phi \phi_x dt \\
=&& \displaystyle\mu |\phi|^{p-2}\phi \phi_{xx} dt + \mu |\phi|^{p-2}\phi \bar u_{xx} dt + |\phi|^{p-2}\phi \sigma (\phi_x + \bar u_x) dB(t),\label{lp}
\end{array}
 \end{eqnarray}
and 
\begin{eqnarray*}
\int_\mathbb{R} |\phi|^{p-2}\phi ( \phi \bar{u} )_x dx &=& \int_\mathbb{R} |\phi|^p \bar{u}_x dx + \int_\mathbb{R}  |\phi|^{p-2}\phi\phi_x \bar{u} dx = \frac{p-1}p \int_\mathbb{R} |\phi|^p \bar{u }_x dx ,\\
\int_\mathbb R |\phi|^{p-2}\phi \phi_{xx} dx &=& \int_\mathbb R |\phi|^{p-2} [ (\phi \phi_x)_x - \phi_x^2] dx = -(p-1) \int_\mathbb R |\phi|^{p-2} \phi_x ^2 dx ,\\
\int_\mathbb{R} |\phi|^{p-2}\phi \bar u_{xx} dx &=& -(p-1) \int_\mathbb R |\phi|^{p-2} \phi_x \bar u_x dx \leq \epsilon \int_\mathbb R |\phi|^{p-2} \phi_x^2 dx + \epsilon \int_\mathbb R |\phi|^p \bar u_x dx + C_\epsilon \int_\mathbb R \bar u_x ^{\frac12 p+1} dx \\
&\leq & \epsilon \int_\mathbb R |\phi|^{p-2} \phi_x^2 dx + \epsilon \int_\mathbb R |\phi|^{p} \bar u_x dx + C_\epsilon (2+t)^{-\frac{p}{2}} .
\end{eqnarray*}
On the other hand,  set $F_p(u) = \int_\mathbb R |u|^p(x) dx $, then $F_p \in \mathbb C^2 \big( \mathbb H^1 (\mathbb R) , \mathbb R \big)  $ and it's first and second derivatives are
\begin{eqnarray*}
   DF_p(u) &=& p |u|^{p-2}u \ \in \mathbb L^2(\mathbb R) ,\\
	D^2F_p(u) &=& (p^2-p) |u|^{p-2} \ \in \mathscr L(\mathbb L^2(\mathbb R)).
\end{eqnarray*}
Then the It\^{o} formula implies that 
\begin{eqnarray}
\begin{array}{lll}
	dF_p( \phi )
	&=& \displaystyle\langle DF_p(\phi), d\phi \rangle + \frac{1}{2} \left \|[ D^2F_p(\phi)]^\frac{1}{2} \sigma(\phi_x + \bar u_x)  \right \|^2dt \\
	  &=& \displaystyle p \langle |\phi|^{p-2}\phi, d\phi \rangle + \frac12(p^2-p)\sigma^2 \left\|  |\phi|^{\frac{p}2-1}  (\phi_x + \bar u_x) \right\|^2dt \\
	  &\leq &\displaystyle  p \langle |\phi|^{p-2}\phi, d\phi \rangle + (\frac12(p^2-p) \sigma^2+\epsilon) \left\| |\phi|^{\frac{p}2-1} \phi_x \right \|^2dt + C_\epsilon \left \| |\phi|^{\frac{p}2-1} \bar u_x \right \|^2dt,
	  \end{array}
\end{eqnarray}
and
\begin{eqnarray}
\left \| |\phi|^{\frac{p}2-1} \bar u_x \right \|^2   &\leq& \epsilon \int_\mathbb R |\phi|^p \bar u_x dx + C_\epsilon \int_\mathbb R \bar u_x ^{\frac{p}2+1} dx  \leq  \epsilon \int_\mathbb R |\phi|^p \bar u_x dx + C_\epsilon (2+t)^{-\frac{p}2} .
\end{eqnarray}
Collecting all estimates above, we get that 
\begin{eqnarray*}
&&\displaystyle d\| \phi \|_{\mathbb L^p(\mathbb R)}^p + ( p-1 - \epsilon ) \int_\mathbb R |\phi|^p \bar u_x dx dt + \frac12(p^2-p)(2\mu - \sigma^2 -\epsilon ) \int_\mathbb R |\phi|^{p-2} \phi_x^2 dx dt \\
&\leq&\displaystyle C_\epsilon (2+t)^{-\frac{p}2}dt + p\sigma \int_ \mathbb R |\phi|^{p-2}\phi \bar u_x dx dB(t).
\end{eqnarray*}
Note that $\sigma^2 < 2\mu$, choosing $\epsilon$ small gives \eqref{3.17}. Thus the proof of Lemma \ref{lemma4.3} is completed.
\end{proof}

\section{Decay estimates}\label{decay estimates}

This section is devoted to the decay rate of strong solution toward the rarefaction wave given in \eqref{rarefaction}. 
\subsection{Decay rate in expectation}

\begin{theorem}\label{t1}
Let $\phi\in X_T $ be the unique strong solution of \eqref{Perturbation}, it holds that for any $0\le t\le T$, 
\begin{eqnarray} \label{l2}
\mathbb E \| \phi(t) \|^2 +\mathbb E\int_0^t\|\phi_x\|^2dt+\mathbb E\int_0^t\int_{\mathbb R}\phi^2\bar{u}_xdxdt\leq  C \ln(2+t), 
\end{eqnarray}
and 
\begin{eqnarray}\label{5.20}
\mathbb{E} \| \phi \|_{\mathbb{L}^p (\mathbb{R})}^p \le C (2 + t)^{- \frac{p-2}{4}} \ln^{p} (2+ t).
\end{eqnarray}
\end{theorem}
\begin{proof}
The $L^2$ estimate \eqref{l2} for $\phi$ directly follows from Lemma \ref{lemma4.1} by taking expectation. For $L^p$ estimate \eqref{5.20}, 
 multiplying \eqref{3.17} by  $(2+t)^\frac{p}{2}$ gives that  

\begin{eqnarray} \label{5.21}
\begin{array}{lll}
	&&\displaystyle d(2+t)^\frac{p}{2} \| \phi \|_ {\mathbb L^{p}(\mathbb R)} ^p + (2+t)^\frac{p}{2}  \int_\mathbb R |\phi|^p \bar u_x + |\phi|^{p-2} \phi_x^2 dx dt \\
	 &\leq& \displaystyle C_1 dt + C_1(2+t)^{\frac{p}{2}-1}\| \phi \|_ {\mathbb L^{p}(\mathbb R)} ^p dt + C_2 (2+t)^\frac{p}{2} \int_ \mathbb R |\phi|^{p-2} \phi  \bar u_x dxdB(t).
	 \end{array}
\end{eqnarray}
Thanks the inequality \eqref{ii1}, we have 
\begin{eqnarray}\label{gn2}
\begin{array}{lll}
\displaystyle	 (2+t)^{\frac{p}{2}-1}\| \phi \|_ {\mathbb L^p(\mathbb R)}^p &\leq & \displaystyle C_p (2+t)^{\frac{p}{2}-1} \|  (\phi^{\frac{p}{2}})_x  \|_ {\mathbb L^2(\mathbb R)}^{\frac{2p-4}{p+2}} \| \phi \|_ {\mathbb L^2(\mathbb R)}^{\frac{4p}{p+2}}\\
	  &\leq & \displaystyle \frac12 (2+t)^\frac{p}{2} \int_\mathbb R |\phi|^{p-2} \phi_x^2 dx + C_p (2+t)^{\frac{p-2}4} \| \phi \|^p.
	  \end{array}
\end{eqnarray}
Integrating (\ref{5.21}) over $[0,t]$ gives that  
\begin{equation} \label{5.22}
	 (2+t)^\frac{p}{2} \| \phi(t) \|_ {\mathbb L^p(\mathbb R)} ^p  \leq C_1t  + C_1 \int_0^t (2+s)^{\frac{p-2}4} \| \phi \|^p ds + C_2 \int_0^t (2+s)^\frac{p}{2} \int_ \mathbb R |\phi|^{p-2} \phi \bar u_x  dx dB(s).
\end{equation}
Taking expectation, we have 
\begin{equation} \label{5.22new}
 \mathbb E\| \phi(t) \|_ {\mathbb L^p(\mathbb R)} ^p  \leq C(2+t)^{\frac{2-p}{2} } + C (2+t)^{-\frac{p}{2}} \int_0^t (2+s)^{\frac{p-2}4} \mathbb E\| \phi \|^p ds.
\end{equation}
We estimate $\mathbb E \| \phi(t) \|^p $ through $\mathbb E \| \phi(t) \|^4 $ and  the BDG inequality.  In fact, \eqref{3.14} gives that  

\begin{eqnarray}
	\| \phi(t) \|^4 & \leq & \left [ C_1 \left [(1+ \| \phi_0 \|^2) + \ln (2+t) \right] + C_2\int_0^t \int_ \mathbb R \phi \bar u_x dx dB(s) \right ]^2,
\end{eqnarray}
and
\begin{eqnarray}\label{e4}
\begin{array}{lll}
	\mathbb E \| \phi \|^4 &\leq  &
\displaystyle C_1 \ln^2(2+t) +  C_2 \mathbb E \int_0^t \left(\int_ \mathbb R \phi \bar{u}_x dx\right)^2 ds\\
 &=& \displaystyle C_1 \ln^2(2+t) +  C_2 \mathbb E \int_0^t \left(\int_{u_-}^{u_+} \phi d\bar{u}\right)^2 ds\\
 &\le &\displaystyle C_1 \ln^2(2+t) +  C_2 (u_+-u_-)\mathbb E \int_0^t \int_{\mathbb R}\phi^2 \bar{u}_xdxds\\
 &\le &\displaystyle C\ln^2(2+t). 
 \end{array}
\end{eqnarray}
We can check that $\mathbb E\|\phi\|^p\le C\ln^{\frac{p}{2}}(2+t), \forall p\in [2,4]$ by the H\"older inequality. Next we shall use the BDG inequality and the decay property $\bar{u}_x\le \frac{C}{2+t}$ to estimae $\mathbb E\|\phi\|^p$ for any $p>4$.

We claim that 
\begin{equation}\label{claim}
\mathbb E\|\phi\|^p\le C_p\ln^{p}(2+t),\quad \forall p\ge 4.	
\end{equation}
The claim will be proved by the following induction principle: find a sequence $p_n\to \infty$ so that \eqref{claim} is valid for all $p_n$. Since \eqref{claim} holds for $p_0=4$ due to \eqref{e4}, we can suppose that \eqref{claim} is true for  $p_k, k=0,1,\cdots,n$ until $p_n$ starting with $p_0=4$. Now we try to find $p_{n+1}>p_n$. Note that $\|\phi\|^2$ is semi-martingale, the BDG inequality and \eqref{3.14} yield that 
\begin{equation}\label{bdg1}
\mathbb E\|\phi\|^{p_{n+1}}\le 	\mathbb E(\sup_{0\le s\le t}\|\phi\|^2)^{\frac{p_{n+1}}{2}}\le  C\left(\ln^{\frac{p_{n+1}}{2}} (2+t)+
\mathbb E\left[ \int_0^t \left(
\int_{\mathbb{R}} \phi \overline{u_{}}_x d x \right)^2 d s
\right]^{\frac{p_{n+1}}{4}}\right)
	\end{equation}
and for some $\frac12<\beta_n<1$ determined later, 
\begin{equation}\label{ep}
	\begin{aligned}
		& \mathbb{E}\left[\int_{0}^{t}\left(\int \phi \bar u_{x} d x\right)^{2} d x\right]^{\frac{p_{n+1}}{4}} 
		\le \mathbb{E}\left[\int_{0}^{t}\left(\int |\phi|^{\beta_{n}} (|\phi|^{1-\beta_{n}}\bar  u_{x}^{\frac12})\bar  u_{x}^{\frac12} d x\right)^{2}\right]^{\frac{p_{n+1}}{4}}\\
		\le&\mathbb E \left[\int_{0}^{t}\left(\int |\phi|^{2(1-\beta_{n})} \bar u_{x} d x\right)\left\|\bar u_{x}\right\|_\frac{1}{1-\beta_{n}}\|\phi\|^{2 \beta_{n}}d s\right]^{\frac{p_{n+1}}{4}}\\
		\leq & \mathbb E \left\{\int_{0}^{t}\left(\int_{R}|\phi|^{2(1-\beta_{n})} \bar u_{x} d x\right)^{\frac{1}{1-\beta_n}} d s+C\left[\int_{0}^{t}(2+s)^{-1} \| \phi\| ^2 d s\right]^{\frac{\beta_np_{n+1}}{4-(1-\beta_n)p_{n+1}}}\right\}\\
		\leq& C\mathbb E \int_{0}^{t} \int_{\mathbb{R}}\phi^{2}\bar u_{x} d x d s+C\mathbb E \left[\int_{0}^{t}(2+s)^{-1} \| \phi\|^{2} d s\right]^{\frac{\beta_np_{n+1}}{4-(1-\beta_n)p_{n+1}}}\\
		\leq& C\ln(2+t)+\ln^{\frac{\beta_np_{n+1}}{4-(1-\beta_n)p_{n+1}}}(2+t)\mathbb E \sup_{0\le s\le t}\|\phi\|^{\frac{2\beta_np_{n+1}}{4-(1-\beta_n)p_{n+1}}}(s).
	\end{aligned}
\end{equation}
Choosing $\beta_n=1-\frac{2}{3p_n}$ and 
\begin{equation}
	p_{n+1}:=\frac{3p_n}{2-\frac{1}{p_n}}\ge \frac{3}{2}p_n
\end{equation}
so that $\frac{2\beta_np_{n+1}}{4-(1-\beta_n)p_{n+1}}=p_n$, we obtain from \eqref{claim}-\eqref{ep}  that 
\begin{equation}
\mathbb E\|\phi\|^{p_{n+1}}\le 	C(\ln^{\frac{p_{n+1}}{2}}(2+t)+\ln^{\frac{p_n}{2}}(2 + t)
\mathbb E\|\phi\|^{p_n}) \le C\ln^{\frac{3}{2}p_n}(2+t)\le C\ln^{p_{n+1}}(2+t).
\end{equation}
Thus we find the sequence $\{p_n\}, n=1,\cdots$ and verify the claim \eqref{claim} for all $p_n$. The claim for $p\in (p_n,p_{n+1})$ can be justified by the H\"older inequality. Hence \eqref{claim} is indeed true for all $p>2$. 

Substituting \eqref{claim} into \eqref{5.22new} gives \eqref{5.20}. Therefore the proof is complete. 
 \end{proof}

 \begin{lemma}
Assume $\sigma^2 < 2 \mu$. Let $\phi\in X_T$ be the solution of (\ref{Perturbation}), then it holds that for any $\epsilon>0$ and any $0\le t\le T$, 
 \begin{eqnarray} \label{5.25}
\mathbb E \| \phi_x(t) \|^2 \leq C_\epsilon (2+t)^{-\frac{1}{2}+\epsilon}.
\end{eqnarray}
 \end{lemma}
 \begin{proof}
Take expectation on (\ref{3.16}), we have
\begin{equation} \label{derivative}
 	d \mathbb E \| \phi_x \|^2 + \mathbb E \| \phi_{xx} \|^2 dt  \leq  C (2+t)^{-2} (1+\mathbb E\|\phi\|^2)dt + C \mathbb E \| \phi \|_{\mathbb L^6(\mathbb R)}^6 dt.
 \end{equation}
Choosing $p=6$ in (\ref{5.20}) gives that
\begin{eqnarray} \label{5.26}
\frac{d}{dt}  \mathbb E \| \phi_x(t) \|^2  \leq C (2+t)^{-1}\ln^6(2+t)\le C_\epsilon(2+t)^{-1+2\epsilon}.
\end{eqnarray}
Note that 
\begin{equation}
\int_0^t\mathbb E\|\phi_x\|^2dt\le C\ln (2+t),  
\end{equation}
the area inequality (Theorem 1.3, $f(t)=\mathbb E\|\phi_x\|^2$) implies that
\begin{eqnarray*}
\mathbb E \| \phi_x(t) \|^2 \leq C_\epsilon (2+t)^{-\frac{1}{2}+\epsilon }.
\end{eqnarray*}
 \end{proof}

%
\begin{theorem}\label{t2}
 Assume $\sigma^2 < 2 \mu$. Let $\phi\in X_T$ be the solution of (\ref{Perturbation}),
 then for any $\epsilon>0$, 
 \begin{eqnarray}
\mathbb E \| \phi \|_{\mathbb L^\infty (\mathbb R)}  \leq  C_\epsilon (2+t)^{-\frac14+\epsilon}.
\end{eqnarray}
 \end{theorem}
 \begin{proof}
 	Thanks the G-N inequality, we have that
 \begin{eqnarray*}
 \| \phi \|_{\mathbb L^\infty (\mathbb R)} \leq C_p  \| \phi \| _{\mathbb L^p (\mathbb R)} ^{\frac{p}{p+2}} \| \phi_x  \| ^\frac{2}{p+2},
\end{eqnarray*}
which gives that 
\begin{eqnarray}
\begin{array}{lll}
\displaystyle \mathbb E \| \phi \|_{\mathbb L^\infty (\mathbb R)} &\leq& C_p \mathbb E \Big(\| \phi \|_ {\mathbb L^{p} (\mathbb R)}^{\frac{p}{p+2}} \| \phi_x \|^\frac{2}{p+2}\Big) \leq C_p \left(\mathbb E \| \phi \|_ {\mathbb L^{p} (\mathbb R)}^p\right) ^{\frac{1}{p+2}}\Big(\mathbb E \| \phi_x \|^2\Big)^{\frac{1}{p+2}}\\
&\leq&   \displaystyle C_{\epsilon}(2+t)^{-\frac{1}{4}+\epsilon},
\end{array}
\end{eqnarray}
by choosing $p$ sufficiently large. Thus the proof is completed. 

 \end{proof}
 \subsection{Decay rate a.s.}
 
 \begin{lemma}
 	Let $\phi\in X_T$ be the strong solution of \eqref{Perturbation}, it holds that for any $p > 2$, there exists a $\mathscr F_\infty$ measurable random variable  $C_p(\omega)\in \mathbb L^2(\Omega)$ such that 
 	\begin{eqnarray}\label{lpas}
 		\| \phi \|_p^p & \leq & C_p(\omega) (2 + t)^{- \alpha}, ~a.s.~
 		\forall \alpha < \frac{p - 2}{4}.
 	\end{eqnarray} 
 \end{lemma}
 
 \begin{proof} For any $\epsilon>0$, multiply \eqref{4.6} by $(2+t)^{-\epsilon}$, one has that 
 \begin{eqnarray}\label{5.16}
 \begin{array}{lll}
 &&\displaystyle \frac{1}{2} d(2+t)^{-\epsilon}\| \phi \|^2 + \frac{\epsilon}{2} (2+t)^{-\epsilon-1}\|\phi\|^2dt+(2+t)^{-\epsilon}\Big[\mu\| \phi_x \|^2 d t 
 +\frac{1}{2}\int_{\mathbb R}\phi^2
 \bar{u}_x d x d t \Big]\\& =& 
\displaystyle (2+t)^{-\epsilon}  \int_{\mathbb R} \phi_{} \bar{u}_{x x} dxdt + \frac{1}{2} (2+t)^{-\epsilon} \sigma^2\|
 \phi_x +\bar{u}_x \|^2 d t + (2+t)^{-\epsilon}\int_{\mathbb R} \phi \bar{u}_x d x d B (t) 
 \end{array}
 \end{eqnarray}
 which gives that 
  \begin{eqnarray}\label{5.17}
 \mathbb E\int_0^t(2+s)^{-\epsilon} \int_{\mathbb R} \phi^2
  \bar{u}_x d x d s\le C_\epsilon. 
  \end{eqnarray}
 Define $M^\epsilon(t) = \int_0^t (2+s)^{-\epsilon}\int_{\mathbb R} \phi \bar{u}_x dxdB (s)$, we have
 \begin{eqnarray}\label{5.18}
 \begin{array}{lll}
 E M^\epsilon(t)^2 & = & \displaystyle \mathbb E \int_0^t (2+s)^{-2\epsilon}\left( \int_{\mathbb R} \phi (x) \bar{u}_x (x) d x \right)^2
 d s\\
&  =& \displaystyle\mathbb E \int_0^t (2+s)^{-2\epsilon} \left( \int_{u_-}^{u_+} \phi (\bar{u}) d \bar{u} \right)^2
 d s \\
 & \leq &\displaystyle (u_+ - u_-) \mathbb E\int_0^t  (2+s)^{-2\epsilon}\int_{u_-}^{u_+} \phi^2 (\bar{u}) d \bar{u}
 d s\\
 &=&\displaystyle(u_+ - u_-) \mathbb E \int_0^t (2+s)^{-2\epsilon}\int_{\mathbb R} \phi^2 \bar{u}_x d x d s\le C_\epsilon,
 \end{array} \label{33} 
 \end{eqnarray}
 which implies from Doob's $\mathbb L^p$ inequality that there exists a $\mathscr F_\infty$ measurable random variable  $C_\epsilon(\omega)\in \mathbb L^2(\Omega)$ such that 
 $$
 |M^\epsilon(t)|\le C_\epsilon(\omega), ~a.s.
 $$
Again using \eqref{5.16} yields that 
 \begin{eqnarray}
 \| \phi \|^2  \leq C_\epsilon (2+t)^\epsilon(1 + M^\epsilon(t))\le C_\epsilon(\omega) (2+t)^\epsilon, ~a.s. \label{44} 
 \end{eqnarray}

On the other hand,  for any $0<\alpha < \frac{p - 2}{4}$, 
%
%
 define
 \begin{eqnarray}
 N_{\alpha, p} (t) & = & \int_0^t (2 + s)^{\alpha} \int_{\mathbb R} | \phi |^{p - 2}
 \phi \bar{u}_x d x d B (s). 
 \end{eqnarray}
 A direct computation gives that 
 \begin{eqnarray}
 \begin{array}{lll}
\displaystyle \mathbb E [N_{\alpha, p}^2 (t)] & = & \displaystyle\mathbb E \int_0^t (2 + s)^{2 \alpha} \left( \int_{\mathbb R} |
 \phi |^{p - 2} \phi \bar{u}_x d x \right)^2 ds r\\
 & = & \displaystyle\int_0^t (2 + s)^{2 \alpha} \mathbb E \left( \int_{u_-}^{u_+} | \phi
 (\bar{u}) |^{p - 2} \phi d \bar{u} \right)^2 ds \\
 & \le & \displaystyle C \int_0^t (2 + s)^{2 \alpha}  \mathbb E\int_{\mathbb R} | \phi |^{2 p - 2} \bar{u}_x
 d x d s\\
 &\le & \displaystyle C\int_0^t (2+t)^{2\alpha-1}(2+t)^{\frac{2-p}{2}}\ln^{2(p-1)}(2+t)dt\le C, 
 \end{array} 
 \end{eqnarray}
 where we have used the fact that 
 $\mathbb E \| \phi (t) \|_{2p-2}^{2p-2}\le C(2+t)^{\frac{2-p}{2}}\ln^{2(p-1)}(2+t)$ due to \eqref{5.20} and $|\bar{u}_x|\le C(2+t)^{-1}$. 
 Then there exists a $\mathscr F_\infty$ measurable random variable  $C_p(\omega)\in \mathbb L^2(\Omega)$ such that 
 $$
 |N_{\alpha,p}(t)|\le C_p(\omega),~a.s.
 $$
Multiplying \eqref{3.17} by $(2+t)^\alpha$ and using \eqref{gn2} and \eqref{44} imply that 
 \begin{eqnarray}
 \begin{array}{lll}
 &&\displaystyle d (2+t)^{\alpha} \| \phi \|_p^p + (2 + t)^{\alpha} \int_{\mathbb R} | \phi |^p
 \bar{u}_x dxdt+(2+t)^\alpha\int_{\mathbb R}|\phi|^{p-2}\phi_x^2dxdt\\
  & \leq & \displaystyle C_1 (2 + t)^{\alpha - \frac{p}{2}}dt + C_1 (2 +
 t)^{\alpha - \frac{p + 2}{4}} \| \phi \|^p + C_2 (2 + t)^{\alpha} \int_{\mathbb R} | \phi |^{p - 2} \phi \bar{u}_x d x d B
 (t)\\
   & \leq & \displaystyle C_1 (2 + t)^{\alpha - \frac{p}{2}}dt + C_1 (2 +
   t)^{\alpha - \frac{p + 2}{4}+\frac{p\epsilon}{2}}C_\epsilon(\omega) + C_2 (2 + t)^{\alpha} \int_{\mathbb R} | \phi |^{p - 2} \phi \bar{u}_x d x d B
   (t),~a.s.
   \end{array}
 \end{eqnarray}
 Integrating over $[0, t]$ and  choosing $\epsilon$ sufficiently small so that $\alpha - \frac{p + 2}{4}+\frac{p\epsilon}{2}<-1$, one has that 
 \begin{eqnarray}
 (2 + t)^{\alpha} \| \phi \|_p^p 
  \leq C (1+C(\omega)) + C
 N_{\alpha, p} (t)\le C_p(\omega),~a.s. \label{8} 
 \end{eqnarray}
 which gives \eqref{lpas}. Thus the proof is completed. 
 \end{proof}
  \begin{lemma}
  	Let $\phi\in X_T$ be the strong solution of \eqref{Perturbation}, it holds that for any $\epsilon>0$, there exists a $\mathscr F_\infty$ measurable random variable  $C_\epsilon(\omega)\in \mathbb L^2(\Omega)$ such that  
  	\begin{eqnarray}\label{deriva2}
  	\| \phi _x\|^2  \leq  C_\epsilon(\omega)(2 + t)^{\epsilon}, ~a.s.
  	\end{eqnarray} 
  \end{lemma}
\begin{proof}  
   Multiplying \eqref{deriva} by $(2+t)^{-\frac12}$, we have  \begin{eqnarray}\label{deriva1}
   \begin{array}{lll}&&	\displaystyle d(2+t)^{-\frac12}\| \phi_x \|^2 + (2+t)^{-\frac12}\| \phi_{xx} \|^2 dt\\[2mm]
   &\leq &\displaystyle C_1(2+t)^{-\frac12} \left((2+t)^{-2} +(2+t)^{-2}\|\phi\|^2+ \| \phi \|_{\mathbb L^6(\mathbb R)}^6 \right)dt\\
   && \displaystyle-C_2 (2+t)^{-\frac12} \int_\mathbb R \phi_{xx}  \bar u_{x}dx dB(t),
   \end{array}
   \end{eqnarray}
 which implies that 
 $$
 \mathbb E\int_0^t(2+s)^{-\frac12}\|\phi_{xx}\|^2dxdt\le C. 
 $$  
%
%
   Define $N(t) = -\int_0^t \int_{\mathbb R} \phi_{x x} \bar{u}_x d x d B (s)$, then in the same way as \eqref{5.18}, one has that 
   \begin{eqnarray}
   \begin{array}{lll}
  \mathbb E N(t)^2&=& \displaystyle\mathbb E \int_0^t \left( \int_{\mathbb R} \phi_{x x} (x) \bar{u}_x (x) d x
   \right)^2 d s 
   \leq  (u_+ - u_-) \int_0^t \| \bar{u}_x \|_{\infty} \mathbb E \int_{\mathbb R} \phi^2_{x
   	x} (x) d x d s \\
  &\le & \displaystyle C\mathbb E \int_0^t  (2+s)^{-\frac12}  \int_{\mathbb R} \phi^2_{x
   	x} d x d s
   \leq  C \label{4}.
   \end{array} 
   \end{eqnarray}
Thus there exists a $\mathscr F_\infty$ measurable random variable  $C(\omega)\in \mathbb L^2(\Omega)$ such that 
$$
|N(t)|\le C(\omega), ~a.s.
$$
Integrating \eqref{deriva} on $[0,t]$, and choosing $p=6$ and $\alpha=-1+\epsilon$ ($\epsilon>0$ is any constant) in \eqref{8}, we get \eqref{deriva2}. The proof is completed. 
 \end{proof}

 \begin{theorem}\label{t3}
 	Let $\phi\in X_T$ be the solution of (\ref{Perturbation}),
 	then for any $\epsilon>0$, there exists a $\mathscr F_\infty$ measurable random variable  $C_\epsilon(\omega)\in \mathbb L^2(\Omega)$ such that 
 	\begin{eqnarray}
 	 \| \phi \|_{\mathbb L^\infty (\mathbb R)}  \leq  C_\epsilon(\omega) (2+t)^{-\frac14+\epsilon},~a.s.
 	\end{eqnarray}
 \end{theorem}
 \begin{proof}
 	From the G-N inequality, we immediately   obtain from \eqref{lpas} and \eqref{deriva2}  that
 	\begin{eqnarray*}
 		\| \phi \|_{ L^\infty (\mathbb R)} \leq C_p  \| \phi \| _{L^p} ^{\frac{p}{p+2}} \| \phi_x  \| ^\frac{2}{p+2}\le C_\epsilon(\omega) (2+t)^{-\frac14+\epsilon}, ~a.s. 
 	\end{eqnarray*}
 by choosing $p$ sufficiently large. The H{\"o}lder inequality implies that $C_\epsilon(\omega)\in \mathbb L^2(\Omega)$.
 	
 \end{proof}

 \section{Global existence}\label{global existence}
It remains to prove the global existence of strong solution of the perturbed equation \eqref{Perturbation}. Due to the effect of noise, the $\|\phi\|_{L^\infty(\mathbb R)}$ norm may not be uniformly bounded. We adopt a cut-off technique to prove the  global existence of  \eqref{Perturbation}. That is, we consider the following cut-off equation
 \begin{eqnarray}\label{cutoff}
 \left\{
 \begin{array} {ll}
 d\phi + ( \phi \bar u )_x dt +  \frac{1}{2} [ ( \Pi _m \phi )^2 ]_x dt = \mu \phi_{xx} dt + \mu \bar u_{xx} dt + \sigma( \phi_x  + \bar u_x ) dB(t),

 \quad in\ \ \mathbb{R}\times [0,\infty), \\
 \phi|_{t=0}(x)=\phi_0(x), \quad on \ \ \mathbb{R},\\
 \end{array}
 \right.	
 \end{eqnarray}
 where
 \[ \Pi_m :  \mathbb H^1(\mathbb{R}) \longrightarrow \mathbb{H}^1(\mathbb{R}),\quad 
 f \longmapsto  \min \left \{ m, \| f \|_{\mathbb H^1(\mathbb{R})}\right \} \frac{f}{ \| f \|_{\mathbb H^1(\mathbb{R})}}.
 \]
Before studying the equation \eqref{cutoff}, we first derive a useful lemma to treat the cut-off term.

\begin{lemma}\label{lemma9}
	Let $H$ be some Hilbert space. Define a cut-off mapping $\Pi_m$ as follows,   $$\Pi_mx=\frac{x}{\|x\|_H}\min\Big\{m,\|x\|_H\Big\}, ~\forall x\in H.$$ 
	Then for any $m>0$ and $x, y \in H$, it holds that 
	\begin{equation}
	\left\|\Pi_{m} x-\Pi_{m} y\right\|_H \leq\|x-y\|_H.
	\end{equation}
\end{lemma}
\begin{proof}
	For brevity, we denote $\|\cdot\|_H$ as $|\cdot|$. Without loss of generality, we assume that $|y| \leq|x|$ and $|x|\ne 0$. The proof is divided into three cases. 
	\\
	Case 1: $\bf{|y|\le m\leq|x|.}$ Define $\rho(t)=\left|(1+t) \Pi_{m} x-y\right|^2, t\ge 0$. A direct computation yields that 
	$$
	\begin{aligned}
	\rho^{\prime}(t) 
	&=\frac{d}{d t}\langle(1+t) \Pi_{m} x-y,(1+t) \Pi_{m} x-y\rangle=\frac{d}{d t}\left[(1+t)^{2}m^2+|y|^2-2(1+t)\left\langle\Pi_{m} x, y\right\rangle\right] \\
	&=2\left[(1+t)m^2-\left\langle\Pi_{m} x, y\right\rangle\right]\ge 2 t m^2\geqslant 0,
	\end{aligned}
	$$
	which indicates that $\rho(t)$ is an increasing function, where we have used the fact that $|\left\langle \Pi_mx,y\right\rangle|\le m^2$. Note that $\Pi_my=y$ and $\frac{|x|}{m}\Pi_mx=x$, we have 
	$$
	\begin{aligned}
	|x-y|^2 &=\rho\Big(\frac{|x|}{m}-1\Big)\ge \rho(0)=\left|\Pi_{m} x-\Pi_{m} y\right|^2.
	\end{aligned}
	$$
	Case 2: $\bf{m<|y|\leq|x|}.$ Let $\tilde{x}=\frac{m}{|y|} x$ and  $\tilde{y}=\frac{m}{|y|} y$, then $m=|\tilde{y}|\le |\tilde{x}|$.  Since $\Pi_m x=\Pi_m \tilde{x}$ and $\Pi_m y=\Pi_m \tilde{y}$,  we have from case 1 that,  
	\begin{eqnarray*}
		\left|\Pi_{m} x-\Pi_{m} y\right| =\left|\Pi_{m} \tilde x-\Pi_{m} \tilde y\right| 
		\leq \left| \tilde x- \tilde y\right| = \frac{m}{|y|}|x-y|\leq |x-y|.
	\end{eqnarray*}
	Case 3: $\bf{|y|\leq|x|\leq m.}$ In this case $\Pi_mx=x$, $\Pi_my=y$. It is obvious that $|\Pi_mx-\Pi_my|=|x-y|$. 
	
	Therefore the proof is complete. 
\end{proof}

The local existence of \eqref{cutoff} is given as follows. 
\begin{thm}[\bf Local existence]\label{theorem7.1}
Assume $\sigma^2 < 2\mu$. Then for any $\phi_0 \in \mathbb H^{2}(\mathbb R)$, there exists a time $T(m)> 0 $ such that the cut-off equation $(\ref{cutoff})$ has a unique solution in $X_{T(m)}$.
\end{thm}
\begin{proof}
We use the iteration method to prove the local esistence. 
Given $\phi^{n}\in X_T$, where $T$ will be chosen later,  let  $\phi^{n+1}$ be the unique strong solution of 
\begin{eqnarray} \label{n2.7}
	\left\{
	    \begin{array}{ll}
 d \phi^{n+1}-\mu \phi^{n+1}_{xx}  dt =\sigma \phi_x^{n+1}dB(t) - ( \phi^{n} \bar u )_x dt - \frac{1}{2} [ ( \Pi _m \phi^n )^2 ]_x dt  + \mu \bar u_{xx} dt +  \sigma  \bar u_xdB(t), \\
	\phi^{n+1}(0)= \phi_0, \quad \phi_0 \in \mathbb H^{2}(\mathbb R), \quad s\in [0,T].\\
	    \end{array}
	    \right.
\end{eqnarray}
Since $\phi^n\in X_T$, it is straightforward to check that  $f:=- ( \phi^{n} \bar u )_x - \frac{1}{2} [ ( \Pi _m \phi^n )^2 ]_x + \mu \bar u_{xx}\in  L^2([0,T]\times\Omega, \mathbb H^1(\mathbb R))$ and $g:=\sigma \bar u_x  \in  L^2([0,T]\times\Omega, \mathbb H^2(\mathbb R)) $. 
Thanks Lemma \ref{rs}, $\phi^{n+1}$ belongs to $W_T\subset X_T$. Thus we define a mapping 
\begin{equation}
\mathscr T_T: X_T \longrightarrow X_T.
\end{equation} 
It remains to show that $\mathscr T_T$ is a contracting mapping for a suitably small $T$. 
Let $ \Phi^n (t) = \phi^{n+1}(t) - \phi^n(t) \  ( n\geq 1 )$, then we have the following stochastic heat equation for $\Phi^n$, 
\begin{eqnarray}\label{ndif}
	\left\{
	    \begin{array}{ll}
 d \Phi^{n}-\mu \Phi^{n}_{xx}  dt =  - ( \Phi^{n-1} \bar u )_x dt - \frac12\left[ \big(\Pi _m \phi^{n}\big)^2 - (\Pi _m \phi^{n-1})^2 \right]_x dt + \sigma\Phi^{n}_x dB(t)  , \\
	\Phi^{n}(0)= 0, \quad \phi^{n-1}, \phi^n,\Phi^{n-1} \in X_T. \\
	    \end{array}
	    \right.
\end{eqnarray}
Multiplying (\ref{ndif}) by  $\Phi^{n}$ and integrating the result on $\mathbb R$, we have   
\begin{eqnarray}
\begin{array}{lll}
&&\displaystyle	\int_\mathbb R \Phi^{n} d \Phi^{n} dx + \mu \int_\mathbb R (\Phi^{n}_{x})^2 dx dt \\
&=&\displaystyle -\int_\mathbb R\Phi^{n}( \Phi^{n-1} \bar u )_x dx dt -\int_\mathbb R \Phi^n \left[ \big(\Pi _m \phi^{n}\big)^2 - (\Pi _m \phi^{n-1})^2 \right]_x dx dt,
\end{array}
\end{eqnarray}
which imples from It\^{o} formula that 
\begin{eqnarray}\label{7.7}
\begin{array}{lll}
	&&\displaystyle\frac12 d\|\Phi^{n}\|^2 + \mu \|\Phi^{n} _{x}\|^2dt \\ 
	&=& \displaystyle\frac12 \sigma^2 d\langle \| \Phi^{n} \|^2 \rangle_t -\int_\mathbb R\Phi^{n}( \Phi^{n-1} \bar u )_x dx dt -\int_\mathbb R \Phi^n \left[ \big(\Pi _m \phi^{n}\big)^2 - (\Pi _m \phi^{n-1})^2 \right]_x dx dt.
	\end{array}
\end{eqnarray}
We estimate the right hand side of \eqref{7.7}  term by term. The Cauchy inequality gives that 
\begin{eqnarray*}
\left| \int_\mathbb R\Phi^{n}( \Phi^{n-1} \bar u )_x dx\right| = \left| \int_\mathbb R\Phi^{n}_x\Phi^{n-1} \bar u  dx \right| \leq  \epsilon \| \Phi^{n}_x\|^2 + C_\epsilon \| \Phi^{n-1} \|^2. 
\end{eqnarray*}
In addition, we have 
\begin{eqnarray*}
\left| \int_\mathbb R \Phi^n \left[ \big(\Pi _m  \phi^{n}\big)^2 - (\Pi _m \phi^{n-1})^2 \right]_x dx\right| 
 &=&   \left | \int_\mathbb R \Phi_x^n \left[\Pi _m  \phi^{n} + \Pi _m \phi^{n-1} \right]\left[\Pi _m \phi^{n} - \Pi _m \phi^{n-1} \right] dx  \right | \\
 &\leq&  \left \| \Pi _m  \phi^{n} + \Pi _m \phi^{n-1} \right \|_{\mathbb L^\infty (\mathbb R)} \int_{\mathbb R} \left | \Phi^n_x \left[\Pi _m  \phi^{n} - \Pi _m \phi^{n-1} \right] \right | dx  \\
 &\leq& Cm \| \Phi^n_x \|  \left \| \Pi _m  \phi^{n} - \Pi _m \phi^{n-1} \right\|  \\
  &\leq & \epsilon \| \Phi^n_x \|^2  + C_\epsilon m^2 \| \Phi^{n-1} \|^2_{H^1(\mathbb R)}
\end{eqnarray*}
where we have used the fact that 
\begin{equation*}
\left \| \Pi _m \phi^{n}- \Pi _m \phi^{n-1} \right\|_{H^1(\mathbb R)}\le C\|\phi^n-\phi^{n-1}\|_{H^1(\mathbb R)}
\end{equation*}
due to Lemma \ref{lemma9}.
From It\^{o} formula, the quadratic variation reads
\begin{eqnarray}
	d\langle \| \Phi^{n} \|^2 \rangle_t = \sigma^2 \|  \Phi^{n-1}_x  \|^2 dt.
\end{eqnarray}
Collecting all the estimates above, integrating on $[0,T]\times \Omega$ and choosing $\epsilon$ small enough,  we have that  
\begin{eqnarray}\label{n2.11}
\begin{array}{lll}
&&\displaystyle\mathbb E\sup_{0\le t\le T}\| \Phi^{n} \|^2 + \mathbb E \int_0^T \| \Phi^{n}_x(s) \|^2 ds\\ 
&\leq&\displaystyle Cm^2  \int_0^T \mathbb E\| \Phi^{n-1}(s) \|^2_{H^1(\mathbb R)} ds\le Cm^2T\mathbb E\sup_{0\le t\le T}\| \Phi^{n-1} \|^2_{H^1(\mathbb R)},
\end{array}
\end{eqnarray}
 On the other hand, multiplying (\ref{ndif}) by  $-\Phi^{n}_{xx}$ and integrating the result on $\mathbb R$ yield that  
\begin{eqnarray}
\begin{array}{lll}
	&&\displaystyle\frac12 d \| \Phi^{n}_x\|^2+\mu \|\Phi^{n} _{x} \|^2dt \\ 
	&=& \displaystyle \frac12 \sigma^2 d\langle \| \Phi^{n}_{x} \|^2 \rangle_t + \int_\mathbb R\Phi^{n}_{xx}( \Phi^{n-1} \bar u )_x dx dt +\int_\mathbb R \Phi^n_{xx} \left[ \big(\Pi _m \phi^{n}\big)^2 - (\Pi _m \phi^{n-1})^2 \right]_x dx dt, 
	\end{array}
\end{eqnarray}
It is straightforward to check that 
\begin{eqnarray*}
\left| \int_\mathbb R\Phi^{n}_{xx}( \Phi^{n-1} \bar u )_x dx \right| =\left| \int_\mathbb R\Phi^{n}_{xx} (\Phi^{n-1}_x \bar u +\Phi^{n} \bar u_x) dx \right| \leq  \epsilon \| \Phi^{n}_{xx}\|^2 + C_\epsilon \| \Phi^{n-1} \|_{\mathbb H^1(\mathbb R)}^2 
\end{eqnarray*}
and 
\begin{eqnarray}\label{7.11}
d\langle \| \Phi^n_x \|^2 \rangle_t = \sigma^2\| \Phi^{n-1}_{xx}  \|^2 dt.
\end{eqnarray}
In addition, we have
\begin{eqnarray*}
&&\left| \int_\mathbb R \Phi^n_{xx} \left[ \big(\Pi _m \phi^{n}\big)^2 - (\Pi _m \phi^{n-1})^2 \right]_x dx \right| \\
& \leq&   \left | \int_\mathbb R \Phi^n_{xx}\left[\Pi _m  \phi^{n} + \Pi _m \phi^{n-1} \right]\left[\Pi _m \phi^{n} - \Pi _m \phi^{n-1} \right]_x dx  \right | \\
&& + \left |  \int_\mathbb R \Phi^n_{xx}\left[\Pi _m \phi^{n} + \Pi _m \phi^{n-1} \right]_x \left[\Pi _m  \phi^{n} - \Pi _m \phi^{n-1} \right] dx \right | \\
& :=& J_1 + J_2. 
\end{eqnarray*}
Direct computation and Lemma \ref{lemma9} imply that 
\begin{eqnarray*}
J_1 
& \leq&  Cm \int_\mathbb R \left | \Phi^n_{xx} \frac{\partial}{\partial x} \left[\Pi _m  \phi^{n} - \Pi _m \phi^{n-1} \right] \right | dx \leq  Cm \| \Phi_{xx}^n \| \left \| \Pi _m  \phi^{n} - \Pi _m \phi^{n-1} \right\|_{\mathbb H^1(\mathbb R)}  \\
& \leq& Cm \| \Phi_{xx}^n \| \left \| \Phi^{n-1} \right\|_{\mathbb H^1(\mathbb R)} \leq \epsilon \| \Phi_{xx}^n \|^2   +  C_\epsilon m^2 \| \Phi^{n-1}  \|_{\mathbb H^1 (\mathbb R)} ^2 ,
\end{eqnarray*}
and
\begin{eqnarray*}
 J_2 
& \leq&  \left \| \Pi _m  \phi^{n} - \Pi _m \phi^{n-1} \right \|_{\mathbb L^\infty (\mathbb R)} \int_\mathbb R \left | \Phi^n_{xx} \frac{\partial}{\partial x} \left[\Pi _m \phi^{n-1} + \Pi _m \phi^{n-1} \right] \right | dx \\
& \leq& \left \|\Pi _m  \phi^{n}- \Pi _m \phi^{n-1} \right \|_{\mathbb H^1 (\mathbb R)} \| \Phi_{xx}^n \| \left \| \frac{\partial}{\partial x} \left[\Pi _m (\Phi^{n-1} + \phi^{n-1}) + \Pi _m \phi^{n-1} \right] \right\| \\
& \leq& Cm \left \|\Pi _m  \phi^{n}- \Pi _m \phi^{n-1} \right \|_{\mathbb H^1 (\mathbb R)}  \| \Phi_{xx}^n \|  \\
& \leq& \epsilon \| \Phi_{xx}^n \|^2   + C_\epsilon m^2  \| \Phi^{n-1}  \|_{\mathbb H^1 (\mathbb R)} ^2 ,
\end{eqnarray*}
Collecting all estimates above, we have
\begin{eqnarray}\label{n7.12}
\begin{array}{lll}
&&\displaystyle\mathbb E \sup_{0 \leq s \leq T} \| \Phi_x^n (s) \|^2 + (2\mu - 8\epsilon) \int_0^T \mathbb E\| \Phi_{xx}^n(s) \|^2 ds\\
 \leq &&\displaystyle C_\epsilon m^2 T \mathbb E \sup_{0 \leq s \leq T} \| \Phi^{n-1}(s) \|_{\mathbb H^1(\mathbb R)}^2 + \sigma^2 \mathbb E \int_0^T \| \Phi_{xx}^{n-1} \|^2 ds.
 \end{array}
\end{eqnarray}
Combining (\ref{n2.11}) and (\ref{n7.12}), choosing $\epsilon$ and $T$ small so that $C_\epsilon m^2T $ is sufficiently small, and $\frac{\sigma^2}{2\mu-8\epsilon}<1$, we have
\begin{eqnarray}
\begin{array}{lll}
&&\displaystyle\|\Phi^n\|_T^2=\mathbb E \left [\sup_{0 \leq s \leq T} \| \Phi^n (s) \|_{\mathbb H^1(\mathbb R)}^2 +  \int_0^T \| \Phi_{x}^n(s) \|_{\mathbb H^1(\mathbb R)}^2 ds \right ] \\
&\leq &\displaystyle\frac{\sigma^2}{2\mu-8\epsilon} \mathbb E  \left[ \sup_{0 \leq s \leq T} \| \Phi^{n-1}(s) \|_{\mathbb H^1(\mathbb R)}^2 +  \int_0^T \| \Phi_{x}^{n-1} \|_{\mathbb H^1(\mathbb R)}^2 ds \right]\le \frac{\sigma^2}{2\mu-8\epsilon}\|\Phi^{n-1}\|_T^2 
\end{array}
\end{eqnarray}
which implies that 
 $\mathscr T_T$ is contracting in $X_T$, thus the equation \eqref{cutoff} has a  unique strong solution in $X_T$.
\end{proof}
\begin{theorem}\label{theorem7.2}
Assume that $\sigma^2 < 2 \mu$ and $\phi_0 \in \mathbb H^{2} (\mathbb R)$,  then for any $\bar{T} \geq 0 $, there exists a unique solution of (\ref{cutoff}) in $X_{\bar{T}}$.
\end{theorem}
\begin{proof}
Note that in Theorem \ref{theorem7.1},  the local time $T$ only depens on $m$, but is independent from the initial data. Thus we can extend the local solution to any time $\bar{T}>0$ for the cut off equation \eqref{cutoff}.	
\end{proof}

On the other hand, it is straightforward to check that all a priori estimates in sections \ref{a priori estimates} and \ref{decay estimates} can be obtained for the cut-off equation \eqref{cutoff} in the same way. Thus we conclude from \eqref{44} and \eqref{deriva2} that 
\begin{eqnarray}\label{6.14}
	\sup_{t \geq 0 } \frac{ \| \phi (t) \|_{\mathbb H^1(\mathbb R)} }{(2+t)^{ \epsilon}} < \ +\infty \quad a.s. \quad \forall \epsilon > 0.
\end{eqnarray}
Now we are ready to prove the global existence of strong solution to \eqref{Perturbation}, that is, 
\begin{theorem}[\bf Global existence]\label{global}
	Assume that $\sigma^2 < 2 \mu$ and $\phi_0 \in \mathbb H^{2} (\mathbb R)$. Then for any fixed  $T > 0 $, there exists a unique strong solution of (\ref{Perturbation}) in $ X_T $ a.s .
	
\end{theorem}
\begin{proof}
	Let $ \phi_m $ be the global solution obtained in Theorem \ref{theorem7.2}, set $$\tau_m = \inf \left \{ t \geq 0: \| \phi_m(t) \|_{\mathbb H^1 (\mathbb R)} \geq m \right \}. $$ Notice that $ \phi_m(t) = \phi_n(t) $ for $ m \geq n $ and all $ t \leq \tau_n $. Set $ \phi(t) = \phi_n(t) $ be the solution of (\ref{Perturbation}) as $ t \leq \tau_n $. Then $ \tau_\infty = \infty $ a.s. holds from \eqref{6.14}. Thus we have 
	\begin{eqnarray}\label{5.9}
	\mathbb P \left( \tau_\infty < \infty \right) = \bigcup_{N=1}^\infty \mathbb P \left( \tau_m < N, \  \forall m \in \mathbb Z^+ \right ) = 0.
	\end{eqnarray}
Therefore Theorem \ref{global} is proved. 
\end{proof}

\

\begin{proof}[\bf Proof of Theorem \ref{1d}]
	The global existence of the unique strong solution to \eqref{stochastic} is proved in Theorem \ref{global} and the decay rates \eqref{l2-1}, \eqref{rate} and \eqref{rate for as} are obtained in Theorems \ref{t1}, \ref{t2} and \ref{t3}, respectively. Therefore the proof of Theorem \ref{1d} is completed. 
	\end{proof}
\section{Instability of viscous shock wave}\label{instability1}
\begin{proof}[\bf Proof of Theorem \ref{instability}]
 As shown in \eqref{vs}, $\tilde{u}(\xi), \xi=x-st$ is the viscous shock wave of the deterministic Burgers equation \eqref{1.1} satisfying 
\begin{eqnarray}\label{vs1}
\left\{
\begin{array} {ll}
-s\tilde{u}'+\tilde{u}\tilde{u}'=\nu \tilde{u}'',\\
\tilde{u}(\xi)\to u_\pm,~\mbox{as}~\xi\to \pm\infty.
\end{array}
\right.	
\end{eqnarray}
For simplicity, we consider the case that  $s=0$ (i.e., $u_-=-u_+>0$). In this situation,  $\tilde{u}(x)$ has the following explicit formula, 
\begin{equation}
\tilde{u}(x)=-u_-+\frac{2u_-c}{e^{hx}+c},
\end{equation}
where $h=\frac{u_-}{\nu}>0$, and $c>0$ could be any constant concerning with the shift of shock wave. Moreover, $\tilde{u}(x)$ is monotonically decreasing with respect to $x$, see Figure 3 below in the $x$-$u$ plane. 
\begin{center}
	\includegraphics[scale=0.6]{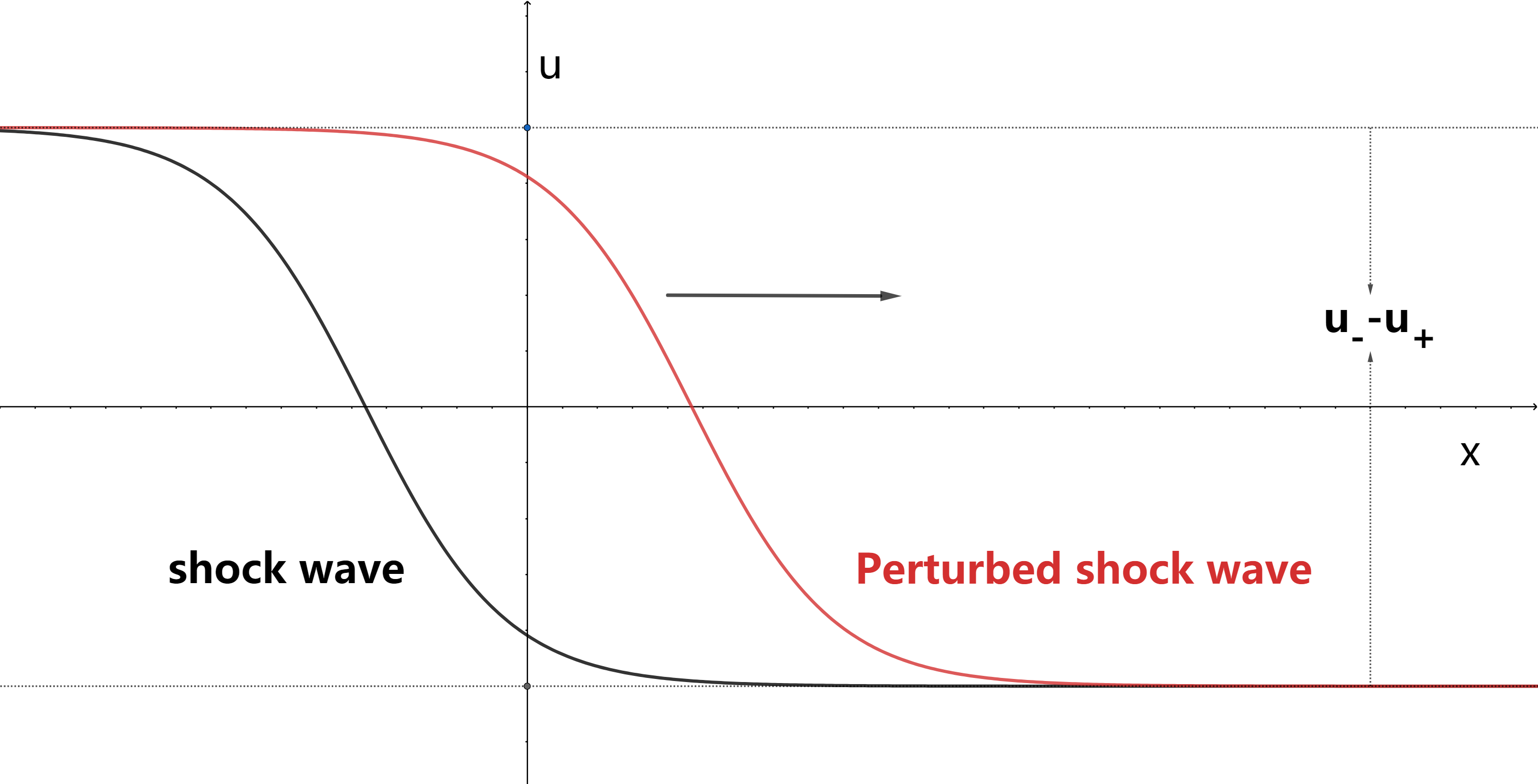}\\[0.1cm]
\end{center}
\begin{center}
	Figure 3
\end{center}
As explained before, the perturbed viscous shock 
is $\tilde{u}^B(t,x):=\tilde{u}(x+\sigma B(t))$ and the two waves coincide at the initial time, i.e., $\tilde{u}(x)=\tilde{u}^B(0,x)$. A direct computation gives that  \begin{eqnarray}
\begin{array}{lll}
d(t)&=:&\displaystyle \mathbb E\|\tilde{u}(x)-\tilde{u}^B(t,x)\|_{L^\infty(\mathbb R)}=\int_{\mathbb R}\|\tilde{u}(x)-\tilde{u}(x+\sigma y)\|_{L^\infty(\mathbb R)}\frac{1}{\sqrt{2\pi t}}e^{-\frac{y^2}{2t}}dy\\
&=&\displaystyle\frac{1}{\sqrt{2\pi }}\int_{\mathbb R}\|\tilde{u}(x)-\tilde{u}(x+\sigma\sqrt{t}z)\|_{L^\infty(\mathbb R)}e^{-\frac{z^2}{2}}dz\\
&=&\displaystyle\frac{1}{\sqrt{2\pi }}\int_{-\infty}^0\|\tilde{u}(x)-\tilde{u}(x+\sigma\sqrt{t}z)\|_{L^\infty(\mathbb R)}e^{-\frac{z^2}{2}}dz
\\
&&\displaystyle+\frac{1}{\sqrt{2\pi }}\int_{0}^{+\infty}\|\tilde{u}(x)-\tilde{u}(x+\sigma\sqrt{t}z)\|_{L^\infty(\mathbb R)}e^{-\frac{z^2}{2}}dz.
\end{array}
\end{eqnarray}
Note that $\tilde{u}(x)$ is monotonically decreasing, it is straightforward to check that for any $z<0$, $\tilde{u}(x+\sigma\sqrt{t}z)$ moves forward as $t$ increases, and  $\|\tilde{u}(x)-\tilde{u}(x+\sigma\sqrt{t}z)\|_{L^\infty(\mathbb R)}$ is monotonically increasing with respect to $t$ and $\lim_{t\to \infty}\|\tilde{u}(x)-\tilde{u}(x+\sigma\sqrt{t}z)\|_{L^\infty(\mathbb R)}=u_--u_+$, see Figure 3. The same argument works for $z>0$.  Thus we conclude from Lebesgue dominated convergence theorem that 
\begin{equation}
\lim_{t\to \infty}d(t)=\frac{1}{\sqrt{2\pi }}\int_{\mathbb R}2u_-e^{-\frac{z^2}{2}}dz=u_--u_+.
\end{equation}
Therefore Theorem \ref{instability} is completed. 
\end{proof}

\section*{Acknowledgements}
Z. Dong is partially supported by National Key R\&D Program of China (No. 2020YFA0712700), 
 Key Laboratory of Random Complex Structure and Data Science, Academy of Mathematics and Systems Science, Chinese Academy of Sciences, NSFC No. 11931004 and  12090014. F. Huang is partially supported by NSFC No. 11688101.

\def\refname{ References}


\begin{thebibliography}{2}
	\bibitem{ABT19}
D. Alonso-Or\'{a}n, A. de Le\'{o}n, S. Takao, \emph{The Burgers' equation with stochastic transport: shock formation, local and global existence of smooth solutions}, Nonlinear Differ. Equ. Appl. vol 26, 1-33(2019).
	
	
\bibitem{cdk}
	G. Chen, Q. Ding and K. Karlsen, \emph{On Nonlinear Stochastic Balance Laws}, Arch. Rat. Mech. Anal., 204, 707-743 (2012).
\bibitem{cg2019} K. Chouk and B. Gess, \emph{Path-by-path regularization by noise for scalar conservation laws}, Journal of Functional Analysis, 277(5), 2019, 1469-1498.
\bibitem{ddt}
G. Da Prato, A. Debussche and R. Temam,
\emph{Stochastic Burgers' equation}, Nonlinear Differential Equations and Applications, Vol.1, 389-402(1994).

\bibitem{dg}
G. Da Prato and D. Gatarek
\emph{Stochastic Burgers equation with correlated noise}, Stochastics and Stochastic Reports, 52, 29-41(1995).

	\bibitem{dv1}
	A. Debussche and J. Vovelle, \emph{Scalar conservation laws with stochastic forcing},  J. Funct. Anal.,259, 1014-1042(2010).
	
	\bibitem{dv2}
	A. Debussche and J. Vovelle, \emph{Invariant measure of scalar first-order conservation laws with stochastic forcing}, Probab. Theory Related Fields, 163, 575-611(2015).

\bibitem{Lillo94}
S. De Lillo,
\emph{The Burgers equation under multiplicative noise}, Physics Letters A, 188(4-6) 305-308(1994).

	
\bibitem{dsxz2018} Z. Dong, X. Sun, H. Xiao and J. Zhai, \emph{Averaging principle for one dimensional stochastic Burgers equation}. J. Differential Equations 265 (2018), no. 10, 4749-4797.
	
 \bibitem{es} W. E, K. Khanin, A. Mazel and Y. Sinai, \emph{Invariant measures for Burgers equation with stochastic forcing}, Ann. Math., 151, 877-960(2000).



\bibitem{fn}
 J. Feng and D. Nualart, \emph{Stochastic scalar conservation laws}, J. Funct. Anal, 255, 313-373(2008). 
 
 \bibitem{f2015} 
 F. Flandoli, \emph{Random perturbations of PDEs and fluid dynamic models}. Lecture Notes in Mathematics 2015, Springer. 
 
\bibitem{fgp2010} F. Flandoli, M. Gubinelli, E. Priola,  \emph{Well-posedness of the transport equation
 by stochastic perturbation}. Inventiones mathematicae 180, 1-53 (2010).
 
 \bibitem{flandoli2021high}
 F. Flandoli, D. Luo, \emph{High mode transport noise improves vorticity blow-up control in 3D Navier--Stokes equations}, Probability Theory and Related Fields,
 180,1,309-363(2021).
 
 \bibitem{g2020} 
  L. Galeati, \emph{On the convergence of stochastic transport equations to a deterministic parabolic one}. 
 	Stoch PDE: Anal Comp (2020) 8: 833-868.
 	
 \bibitem{ggls2020} P. Gassiat, B. Gess, P. Lions and P. Souganidis, 
 \emph{Speed of propagation for Hamilton-Jacobi equations with multiplicative rough time dependence and convex Hamiltonians}, Probability Theory and Related Fields, 176, 421-448 (2020).
 	
\bibitem{gm2018} B. Gess and M. Maurelli, \emph{Well-posedness by noise for scalar conservation laws}, Commun. Partial Differ. Equ., 43(12), 2018, 1702-1736.

\bibitem{gs2019} B. Gess and S. Smith, \emph{Stochastic continuity equations with conservative noise}, Journal de Mathématiques Pures et Appliquées, 128, 2019, 225-263.
 	
 \bibitem{gs2017} B. Gess and P. Souganidis, \emph{Long-time behavior, invariant measures and regularizing effects for stochastic scalar conservation laws}, Comm. Pure Appl. Math., 
70(8), 2017, 1562-1597.
 
 \bibitem{gjs2015} P. Goncalves, M. Jara and S. Sethuraman, \emph{A stochastic Burgers equation from a class of microscopic interactions}. Ann. Probab. 43 (2015), no. 1, 286-338.
 
\bibitem{gj2013} M. Gubinelli and M. Jara,  \emph{Regularization by noise and stochastic Burgers equations}. Stoch. Partial Differ. Equ. Anal. Comput. 1 (2013), no. 2, 325-350.

\bibitem{hm2016} M. Hairer and K. Matetski, \emph{Optimal rate of convergence for stochastic Burgers-type equations}. Stoch. Partial Differ. Equ. Anal. Comput. 4 (2016), no. 2, 402-437.

\bibitem{hw2013} M. Hairer and H. Weber, \emph{Rough Burgers-like equations with multiplicative noise.} Probab. Theory Related Fields 155 (2013), no. 1-2, 71-126.

 
\bibitem{H&N91}
Y. Hattori and K. Nishihara,
\emph{A note on the stabilty of the rarefaction wave of the Burgers equation}, Japan J, Indust. Appl. Math. 8 85-96(1991).
\bibitem{hl}
L. Hsiao and T. Liu, \emph{Convergence to nonlinear diffusion waves for solutions of a system of hyperbolic conservation laws with damping}, Commun. Math. Phys., \textbf{143}, 599-605(1992).
\bibitem{Huang-Xin-Yang}
\newblock F.  Huang, Z.  Xin and T. Yang,
\newblock \emph{Contact discontinuities with general perturbation for gas motion},
\newblock Adv. Math., \textbf{219}, 1246--1297(2008).

\bibitem{hx1} F. Huang and L. Xu, \emph{Decay rate toward the traveling wave for scalar viscous conservation law}, Preprint. 

\bibitem{hx2} F.  Huang and L. Xu, \emph{Decay rate toward the planar rarefaction wave for saclar viscous conservation law in several dimensions}, Preprint. 

\bibitem{kawashima} S. Kawashima, S. Nishibata and M. Nishikawa, \emph{$L^p$ Energy method for multi-dimensional viscous conservation laws and application to the stability of planar waves}, Journal of Hyperbolic Differential Equations, Vol.1, No. 3,  581-603 (2004).

\bibitem{Kr}
N. Krylov, \emph{An Analytic Approach to SPDEs}, Stochastic Partial Differential Equations,ed. by B.L. Rozovskii, R. Carmona. Six Perspectives, Mathematical Surveys and Monographs, American Mathematical Society, Providence, 185-242(1999).  

\bibitem{Olinik64}
A.  Ilin and O. Oleinik,
 \emph{Behavior of the solution of the Cauchy problem for certain quasilinear equations for unbounded increase of time}, Amer. Math. Soc. Transl. Ser. 2 42 19-23(1964).
 
 \bibitem{Ito96}
 K. Ito,
 \emph{Asymptotic decay toward the planner rarefaction waves of solutions for viscous conservation laws in several space dimensions}, Math. Models methods Appl. Sci, 6, 315- 338(1996).
 
 \bibitem{Liu-Yang-Yu-Zhao}
   T. Liu, T. Yang, S.  Yu and H.  Zhao,
  \newblock \emph{Nonlinear stability of rarefaction waves
	for the Boltzmann equation},
  \newblock Arch. Rat. Mech. Anal., \textbf{181}, 333--371(2006).
  
  \bibitem{Liu-Yu}
   T. Liu and S. Yu,
  \newblock \emph{Boltzmann  equation:
  	Micro-macro decompositions  and positivity of shock profiles},
  \newblock Commun. Math. Phys., \textbf{246}, 133-179(2004).
  
  \bibitem{LZ3} T. Liu and Y. Zeng, \emph{Shock waves in conservation laws with physical viscosity}, Mem. Amer. Math. Soc. 234 (2015), no. 1105, vi+168 pp. ISBN: 978-1-4704-1016-2. 
  
  \bibitem{lw2016} G. Lv and J. Wu, \emph{Renormalized entropy solutions of stochastic scalar conservation laws with boundary condition}. J. Funct. Anal. 271 (2016), no. 8, 2308-2338.
  
\bibitem{MN85} A. Matsumura and K. Nishihara, \emph{On the stability of traveling wave solutions of a one-dimensional model system for compressible viscous gas}, Japan J. Appl. Math., 2, 17-25(1985).

\bibitem{M&N86}
A. Matsumura and K. Nishihara,
\emph{Asymptotic towards the rarefaction wave solutions of one-dimensional model system for compressible viscous gas}, Japan J. Appl. Math. 3, 1-13(1986).


\bibitem{MN92} A. Matsumura and K. Nishihara, \emph{Global stability of the rarefaction wave of a onedimensional model system for compressible viscous gas}, Comm. Math. Phys., 144, 325-335(1992).

\bibitem{Nikishiwa99}
M. Nishikawa and K. Nishihara, 
\emph{Asymptotics towards the planar rarefaction wave for viscous conservation law in two space dimensions}, Amer. Math. Soc. Transl. Vol. 352, 1203-1215(2000).

\bibitem{pz2020} X. Peng and R. Zhang, \emph{Approximations of stochastic 3D tamed Navier-Stokes equations}. Commun. Pure Appl. Anal. 19 (2020), no. 12, 5337-5365.


\bibitem{Smoller} 
 J. Smoller,
\emph{Shock Waves and Reaction-Diffusion Equations}, New York: Springer(1994).

\bibitem{Javier04}
J. Villarroel,
\emph{The stochastic Burgers equation in Ito's sense}, Studies. Appl. Math. 1 87-100(2004).

\bibitem{Xin90}
Z. Xin,
\emph{Asymptotic stability of planner rarefaction waves for viscous conservation laws in several dimensions}, Trans. Amer. Math. Soc. 319, 805-820(1990).


 \bibitem{Z} K. Zumbrun, \emph{Stability of large-amplitude shock waves of compressible Navier-Stokes equations}. With an appendix by Helge Kristian Jenssen and Gregory Lyng. Handbook of mathematical fluid dynamics. Vol. III, 311-533, North-Holland, Amsterdam(2004). 

\end{thebibliography}
\end{document}